\documentclass[11pt,reqno,british,a4paper,twoside]{article}

\usepackage[utf8]{inputenc}
\usepackage[T1]{fontenc}
\usepackage{amssymb,tgpagella}
\usepackage[euler-digits,small]{eulervm}
\usepackage{mathtools,babel,microtype}
\usepackage{amscd,xspace,fancyhdr,textcmds,csquotes}
\setquotestyle[british]{english}

\usepackage[inline,shortlabels]{enumitem}

\usepackage[amsmath,thmmarks,hyperref]{ntheorem}

\usepackage[colorlinks=true,allcolors=black,unicode=true]{hyperref}
\usepackage[capitalise,noabbrev]{cleveref}

\rmfamily
\mathtoolsset{mathic=true}
\setlist{nosep}
\setlist[enumerate,1]{\upshape (i)}

\theoremstyle{plain}
\theoremseparator{.}
\newtheorem{theorem}{Theorem}[section]
\newtheorem{proposition}[theorem]{Proposition}
\newtheorem{lemma}[theorem]{Lemma}
\newtheorem{corollary}[theorem]{Corollary}

\theorembodyfont{\normalfont}

\theoremheaderfont{\itshape}
\theoremsymbol{{\small\ensuremath{\quad\triangle}}}

\theoremsymbol{{\small\ensuremath{\quad\diamondsuit}}}

\theoremstyle{nonumberplain}
\theoremheaderfont{\itshape}
\theorembodyfont{\normalfont}
\theoremsymbol{{\small\ensuremath{\quad\Box}}}
\newtheorem{proof}{Proof}

\qedsymbol{{\small\ensuremath{\quad\Box}}}

\numberwithin{equation}{section}
\crefname{equation}{equation}{equations}
\numberwithin{figure}{section}

\crefformat{section}{\S #2#1#3}
\Crefformat{section}{Section~#2#1#3}

\newcommand{\hyphen}{-\hspace{0pt}}

\newcommand{\Lie}[1]{\operatorname{#1}}
\newcommand{\lie}[1]{\operatorname{\mathfrak{#1}}}
\newcommand{\Sp}{\Lie{Sp}}
\newcommand{\sP}{\lie{sp}}
\newcommand{\SU}{\Lie{SU}}
\newcommand{\Un}{\Lie{U}}

\newcommand{\bC}{\mathbb C}
\newcommand{\bH}{\mathbb H}
\newcommand{\bR}{\mathbb R}
\newcommand{\bZ}{\mathbb Z}

\DeclareMathOperator{\codim}{codim}
\DeclareMathOperator{\dist}{dist}
\DeclareMathOperator{\im}{Im}
\DeclareMathOperator{\rank}{rank}

\DeclareMathOperator{\stab}{stab}
\DeclareMathOperator{\vol}{vol}

\newcommand{\hor}{{\mathcal H}}
\newcommand{\Hrel}{\mathrel{\sim_{\hor}^{}}}

\newcommand{\Hodge}{{*}}

\newcommand{\HKT}{\textsc{hkt}\xspace}
\newcommand{\Hrelated}{\( \hor \)\hyphen related\xspace}

\newcommand{\Mmod}{M_{\mathrm{mod}}}
\newcommand{\Mgmod}{M_{\mathrm{mod}N}}
\newcommand{\gmod}{g_{\mathrm{mod}}}
\newcommand{\gHX}{g_{\bH X}}
\newcommand{\gHXc}{g_{\bH\check X}}
\newcommand{\muH}{\mu_\bH}
\newcommand{\XH}{X_\bH}
\newcommand{\hook}{\mathop{\lrcorner}}

\newcommand{\hkq}{{/\mkern-6mu/\mkern-6mu/\mkern1mu}}

\DeclarePairedDelimiter{\abs}{\lvert}{\rvert}
\DeclarePairedDelimiter{\norm}{\lVert}{\rVert}
\DeclarePairedDelimiterX{\inp}[2]{\langle}{\rangle}{#1, #2}
\DeclarePairedDelimiter{\Span}{\langle}{\rangle}
\DeclarePairedDelimiterX{\Set}[2]{\{}{\}}{\, #1 \,\delimsize\vert\, #2 \,}

\newcommand{\eqbreak}[1][2]{\\&\hskip#1em}

\begin{document}
\thispagestyle{empty}

\begin{center}
  \LARGE\bfseries Twists versus Modifications
\end{center}
\begin{center}
  Andrew Swann
\end{center}


\begin{abstract}
  The twist construction is a geometric T-duality that produces new
  manifolds from old, works well with for example hypercomplex
  structures and is easily inverted.  It tends to destroy properties
  such as the hyperKähler condition.  On the other hand modifications
  preserve the hyperKähler property, but do not have an obvious
  inversion.  In this paper we show how elementary deformations
  provide a link between the two constructions, and use the twist
  construction to build hyperKähler and strong HKT structures.  In the
  process, we provide a full classification of complete hyperKähler
  four-manifolds with tri-Hamiltonian symmetry and study a number
  singular phenomena in detail.
\end{abstract}

\bigskip

\begin{center}
  \begin{minipage}{0.8\linewidth}
    \microtypesetup{protrusion=false}
    \small \tableofcontents
  \end{minipage}
\end{center}

\bigskip

\begin{flushleft}
  \par
\end{flushleft}

\newpage
\section{Introduction}
\label{sec:introduction}

HyperKähler metrics are Ricci-flat with a triple of parallel complex
structures.
The metric together with any one of the complex structures
specifies a Kähler geometry with parallel complex-symplectic form.
Such manifolds are Calabi-Yau and form a special class in the Berger
holonomy classification, see Besse~\cite{Besse:Einstein}.

Given an isometric circle action preserving each element in the
triple, there are at least two different constructions that may be
applied to produce manifolds in the same dimension with a new topology
and a metric compatible with a triple of complex structures.

The first is the twist construction of~\cite{Swann:twist}, which
reproduces the T-duality as used in Gibbons, Papadopoulos and Stelle
\cite{Gibbons-PS:hkt-okt}. 
In particular, it includes constructions of strong \HKT metrics in
dimension four from hyperKähler metrics.
However, many of the examples discussed in~\cite{Gibbons-PS:hkt-okt}
are incomplete, and it is not clear whether one can derive hyperKähler
metrics from the construction.
Indeed, while the twist construction in~\cite{Swann:twist} is easily
specialised to generate integrable complex structures,
it does tend to destroy symplectic structures that are present.
On the other hand the twist construction has the advantage that is
a genuine duality and may easily be inverted.

In contrast, the hyperKähler modification construction
\cite{Dancer-S:mod} produces hyperKähler manifolds in the same
dimension via a hyperKähler moment map construction.  
When the original manifold is simply-connected, the modification
increases the section Betti number by one.
Away from the zero set of the moment map, the topological set-up is
precisely a double fibration picture that is the basis for the twist
construction.
However, the recipe for producing the hyperKähler metric from this
picture is rather different, and inversion is not apparent.

The purpose of this paper is to determine exactly how these
constructions are related to each other,
particularly when the metrics involved are complete.
As the domain of the hyperKähler modification is larger than that of
the double fibration, this also enables us to explore some singular
behaviour of the twist construction.

We start the paper by giving a brief overview of the twist and
modification constructions, describing their properties with respect
to completeness.  We also introduce a generalisation of the
modification, which has as an ingredient an arbitrary complete hyperKähler
four-manifold with circle symmetry in place of flat~\( \bH = \bR^4 \)
in the original construction.
Such four-manifolds were classified by
Bielawski~\cite{Bielawski:tri-Hamiltonian} under the assumption that the second
Betti number is finite.  
In \cref{sec:hyperk-four-manif}, we extend his result to the general
case, showing that the only other examples are the
\( A_\infty \)-manifolds of Anderson, Kronheimer and LeBrun \cite{AndersonMT-KL:infinite}
and their Taub-NUT deformations.
The general strategy we use is that of Bielawski's original work,
extending the Gibbons-Hawking Ansatz,
but with a different approach to the analysis of positive harmonic 
functions and invocation of the Martin boundary.
With this in place we make a local study of the general modification
in \cref{sec:twist} and show how on dense open sets
it may be interpreted as a twist:
not of the original hyperKähler metric, but rather of an
\enquote{elementary deformation} obtained by scaling the metric in
quaternionic directions generated by the symmetry.
A similar concept of elementary deformation was used in
\cite{Macia-S:c-map} to describe the hyperKähler\slash quaternionic
Kähler correspondence when there is a symmetry that preserves only one
of the complex structures.

It is now natural to ask when elementary deformations combined with
the twist construction lead to hyperKähler metrics.
\Cref{sec:gener-hyperk-twists} describes this on the principal locus,
showing that these are essentially governed by harmonic functions
on~\( \bR^3 \).
In contrast to the four-dimensional Gibbons-Hawking description,
these harmonic functions are not necessarily positive.
Indeed the twist construction has the virtue of being easily inverted,
and the inversion is seen to be governed by the negative of the
original harmonic function.
This answers the question of how the hyperKähler modification may be
inverted.
In \cref{sec:pseudo-riem-struct}, we briefly discuss how a similar
derivation of which twists are hyperKähler may be obtained in the
pseudo-Riemannian situation.

The harmonic function controlling elementary deformations whose twist
are hyperKähler is allowed to have singularities.
Indeed the original modification construction corresponds to a
\( 1/(2\norm q) \) singularity in~\( \bR^3 \).
In \cref{sec:singular-behaviour}, we dissect what singular behaviour
is allowed when twisting between complete metrics.
A close study of the interaction of the circle symmetry with its
hyperKähler moment map, shows that many of these singularities are of
the type of a standard modification or its inverse. 

Finally in \cref{sec:strong-hkt-metrics}, we study the use of
elementary deformations to twist hyperKähler metrics to strong \HKT
metrics.
For \HKT geometries one has certain type of triple of Hermitian
structures, that are not necessarily symplectic.
The strong condition is that a certain three-form is closed.
Very few examples of such structures are known that are not
hyperKähler: following work of Joyce~\cite{Joyce:hypercomplex}, and
building on Spindel et al. \cite{Spindel-STvP:complex},
Grantcharov and Poon \cite{Grantcharov-P:HKT} showed that bi-invariant metrics on most
compact groups of dimension \( 4n \) are strong \HKT;
Barberis and Fino \cite{Barberis-F:strong} produced other compact examples;
otherwise most known examples, including those in
\cite{Gibbons-PS:hkt-okt}, are incomplete.
We show that starting with a hyperKähler manifold with circle symmetry,
harmonic functions again govern which twists are strong \HKT.
Complete examples, for example on cotangent bundles of complexified
compact Lie groups, may be obtained simply by taking this function to be
constant. 

\paragraph*{Acknowledgements.}
\label{sec:acknowledgements}

I thank Roger Bielawski, Marcel Bökstedt, Andrew Dancer, Andrew du
Plessis, Anna Fino, Marco Freibert, Thomas Bruun Madsen and Bent
Ørsted for useful conversations.
This work is partially supported by the Danish Council for Independent
Research, Natural Sciences.

\section{Fundamental constructions}
\label{sec:set-up}

Let \( M \) be a hyperKähler manifold with metric \( g \) and complex
structures \( I \), \( J \) and~\( K \).
By definition, \( \omega_I = g(I\cdot,\cdot) \) is a Kähler form
for~\( (g,I) \), etc., and \( IJ = K = -JI \).
Suppose that \( M \) has a tri-holomorphic circle action
generated by~\( X \), so \( X \)~is a vector field with \( L_Xg = 0 \)
and \( L_XI = 0 = L_XJ \).
We will assume that the action is effective and has period~\( 2\pi \).
If \( X \) is Hamiltonian for \( \omega_I \), \( \omega_J \) and \(
\omega_K \), we say that the action is \emph{tri-Hamiltonian} and
write \( \mu = \mu^I\mathbf i+\mu^J\mathbf j+\mu^K\mathbf k \in \im\bH
\) for the hyperKähler moment map, where by definition \( \mu^I \in
C^\infty(M) \) satisfies \( d\mu^I = X\hook\omega_I \).

Suppose \( \XH \) is the vector field on \( \bH \cong \bR^4 \) that
generates the circle action \( \mathbf q\mapsto e^{i\theta}\mathbf q \).
This action is tri-Hamiltonian and a corresponding hyperKähler moment
map~\( \muH \) is given by
\begin{equation}
  \label{eq:mu-H}
  \muH(\mathbf q) = \tfrac 12\overline{\mathbf q} i \mathbf q.
\end{equation}
On \( M\times \bH \) we have a circle action generated by \( Y = X -
\XH \).  The hyperKähler quotient \cite{Hitchin-KLR:hK} of \( M\times
\bH \) by~\( Y \) is the \emph{modification} \( \Mmod \) of~\( M \) as
defined in~\cite{Dancer-S:mod}. 

More generally, we may replace \( \bH \) with a complete hyperKähler
four-manifold \( N \) that has a tri-Hamiltonian circle action of
period~\( 2\pi \) generated by~\( X_N \) with hyperKähler moment
map~\( \mu_N \).
We then define the \emph{general modification} \( \Mgmod \) of \( M \)
by~\( N \) by taking
\begin{equation*}
  Y = X - X_N
\end{equation*}
and setting
\begin{equation*}
  \Mgmod = (M \times N) \hkq Y = \rho^{-1}(0)/Y,
\end{equation*}
where \( \rho = \mu - \mu_N\colon M \times N \to \bR^3 \) is the
hyperKähler moment map for~\( Y \).  
This general modification is a smooth hyperKähler
manifold whenever \( Y \) acts freely on~\( \rho^{-1}(0) \).  
For any constant \( c \), we have that \( \mu_N + c \) is also a
hyperKähler moment map on~\( N \); with this choice the quotient is smooth
if \( Y \) acts freely on \( \rho^{-1}(c) \).
We will absorb \( c \) in to \( \mu_N \) and call \( \mu_N \)
a \emph{good} moment map for~\( M \) when the general modification is
smooth. 
Note that we have a double fibration
\begin{equation}
  \label{eq:D0}
  \begin{CD}
    M @<\pi_1<< D_0 @>\pi>> \Mgmod
  \end{CD}
\end{equation}
of principal circle bundles over open dense subsets of \( M \)
and~\( \Mgmod \),
where \( D_0 \subset \rho^{-1}(0) \) is the open dense set on which
both \( X_N \) and~\( Y \) act freely.

On the other hand, for the twist construction
\cite{Swann:twist,Swann:T} starting with \( (M,X) \) one takes a
principal circle bundle \( P\to M \) with a connection~\( \theta \)
such that the curvature~\( F = d\theta \) is invariant under \( X \)
and the contraction \( X\hook F \) is exact.
Choosing \( a\in C^\infty(M) \) with \( da = -X\hook F \), one obtains
an \( \bR \)-action on~\( P \) generated by
\( X' = \widetilde X + aZ \),
where \( Z \) is the generator of the principal action of~\( P \)
and \( \widetilde X \) is the horizontal lift of~\( X \) to~\( \hor =
\ker\theta \).
The \emph{twist}~\( W \) of~\( M \) is then defined to be the
quotient~\( P/\Span{X'} \) when \( X' \) generates a circle action.
Again we have a double fibration 
\begin{equation}
  \label{eq:twist}
  \begin{CD}
    M @<\pi_M<< P @>\pi_W>> W.
  \end{CD}
\end{equation}
Invariant forms and metrics are transferred from \( M \) to~\( W \)
by pulling them back to~\( P \), restricting them to \( \hor \) and
then pushing them down to~\( W \).
This works as long as \( X' \) is transverse to \(\hor \),
which is the same as the function~\( a \) having no zeroes.
Tensors \( \alpha \) on~\( M \) and \( \alpha_W \) on \( \Omega^p(W) \)
related in this way are then said to be \emph{\Hrelated} and we write
\( \alpha_W^{} \Hrel \alpha \).  
In the case of differential forms, their exterior differentials
satisfy
\begin{equation}
  \label{eq:dW}
  d\alpha_W^{} \Hrel d\alpha - \tfrac1a F\wedge (X\hook\alpha).
\end{equation}

The twist construction is easily inverted.
If \( M \) is twisted to \( W \) via \( (X,a,F) \) then
twisting \( W \) by the data \Hrelated to
\( (-\tfrac1aX,\tfrac1a,\tfrac1aF) \) 
recovers~\( M \).

For both constructions, it is possible to control
completeness properties. 

\begin{proposition}
  \leavevmode
  \begin{enumerate}
  \item\label{item:mod-compl} If \( M \) is a complete hyperKähler
    manifold, then any hyperKähler modification \( \Mgmod \) of \( M \)
    is complete.
  \item\label{item:twist-compl} If \( g \) is a complete metric on a
    manifold~\( M \) then any smooth twist \( (W,g_W) \) of \( (M,g)
    \) by an isometry of~\( g \) is complete.
  \end{enumerate}
\end{proposition}

\begin{proof}
  For \ref{item:mod-compl} it is enough to note that hyperKähler
  quotients of complete manifolds by compact groups are always
  complete.  In particular, \( M \)~and \( N \) complete imply \(
  M\times N \) is complete and so the hyperKähler quotient \( \Mgmod =
  (M\times N)\hkq S^1 \) is complete.

  For the twist construction \ref{item:twist-compl}, any curve \(
  \gamma_W \) on~\( W \) may be lifted horizontally to~\( P \).  This
  is then the horizontal lift of a unique curve~\( \gamma \) on~\( M
  \).  As \( g_W \) is \Hrelated to \( g \), we see that \( \gamma \)
  and \( \gamma_W \) have the same lengths.  Since we are twisting by
  a compact group, \( \gamma_W \) lies in a compact subset of~\( W \)
  if and only if \( \gamma \) lies in a compact subset of~\( M \).

  Suppose \( \gamma_W \) lies in no compact subset of \( W \), then \(
  \gamma \) is contained in no compact subset of~\( M \).
  Completeness of \( (M,g) \) implies that \( \gamma \) has infinite
  length, and the same is true of \( \gamma_W \), so \( g_W \) is
  complete.
\end{proof}

\section{HyperKähler four-manifolds with symmetry}
\label{sec:hyperk-four-manif}

Consider a hyperKähler manifold \( (M,g,I,J,K) \) of real
dimension~\( 4 \) with a tri-Hamiltonian circle symmetry generated
by~\( X \).
Define \( M' = M \setminus M^X \) to be the complement of
the set~\( M^X \) of zeros of~\( X \).
Putting \( \alpha_0 = g(X,\cdot) = X^\flat \) and
\( \alpha_I = IX^\flat = X\hook\omega_I \), etc., we have on \( M' \)
that
\begin{equation*}
  g = V(\alpha_0^2 + \alpha_I^2 + \alpha_J^2 + \alpha_K^2),
  \quad
  \omega_I = V(\alpha_{0I}+\alpha_{JK}),\quad\text{etc.},
\end{equation*}
where \( V = 1/g(X,X) \) and
\( \alpha_{0I} = \alpha_0\wedge\alpha_I \).
Introducing \( \beta_0 = V\alpha_0 \), so \( \beta_0(X) = 1 \),
the metric \( g \) becomes
\begin{equation*}
  g = \frac1V\beta_0^2 + V(\alpha_I^2+\alpha_J^2+\alpha_K^2)
\end{equation*}
Now on \( M'/X \), the map \( \mu=(\mu^I,\mu^J,\mu^K) \) is a local
diffeomorphism.  Noting that \( \alpha_I = d\mu_I \) and choosing a
parameter \( t \) along the orbits, the metric \( g \) is locally
\begin{equation}
  \label{eq:G-H}
  g = \frac1V(dt+\omega)^2 + V(dx^2+dy^2+dz^2),
\end{equation}
where  \( \omega = \beta_0 - dt \) is (the pull-back of) a one-form
on an open subset of~\( \bR^3 \).
In this local form, the closure of \( \omega_I \), \( \omega_J \)
and \( \omega_K \) is equivalent to the monopole equation
\begin{equation}
  \label{eq:monopole}
  d\omega = - \Hodge_3dV
\end{equation}
in~\( \bR^3 \), see~\cite{Hitchin:Montreal,Tod-W:sd-Killing}.
Here \( \Hodge_3 \) is the Hodge star operator with respect to the
standard metric \( g_{\bR^3} = dx^2+dy^2+dz^2 \) on~\( \bR^3 \).
In particular, equation~\eqref{eq:monopole} shows that
\( V \)~is locally a harmonic function in~\( \bR^3 \).
Conversely local positive harmonic functions on~\( \bR^3 \) give
four-dimensional hyperKähler metrics via \eqref{eq:monopole}
and~\eqref{eq:G-H}. 

Choosing various harmonic functions leads to a number of examples of
complete hyperKähler four-manifolds with tri-holomorphic circle actions.
\begin{enumerate}
\item For \( V(p) = \frac12 \sum_{i=1}^k \norm{p-p_i}^{-1} \), with \( p_i
  \in \bR^3 \) distinct, we get the Gibbons-Hawking gravitational
  instantons~\cite{Gibbons-H:multi}.
  This has second Betti number \( b_2(M) = k - 1 \).
\item For \( V(p) = \frac12 \sum_{i=1}^\infty \norm{p-p_i}^{-1} \),
  with \( p_i \in \bR^3 \) distinct and so that \( V(q) < \infty \) for some
  \( q \in \bR^3 \), we obtain the metrics of
  Anderson, Kronheimer and LeBrun \cite{AndersonMT-KL:infinite}, which are of infinite topological
  type.
  (They are all complete, since arc-length is bounded below by
  arc-length in~\( \bR^4 \), from the case \( V(p) = 1/(2\norm{p-p_1}) \).)
\item Adding a positive constant \( c \) to either of the previous two
  potentials we get Taub-NUT deformations of the metrics.  In the
  finite case, these are to be found in~\cite{Hawking:gravitational}.
\end{enumerate}
Bielawski \cite{Bielawski:tri-Hamiltonian} proved that the Gibbons-Hawking
metrics and their Taub-NUT deformations are the only complete examples
of finite topological type.  In this section we will show how
Bielawski's arguments may be strengthened to remove this topological
restriction.

Note that Goto \cite{Goto:A-infinity} provided an alternative
construction of the Anderson-Kronheimer-LeBrun metrics
that Hattori \cite{Hattori:c-symplectic-Ainfty,Hattori:Ainfty-volume} has
discussed equivalence problems within this class and volume growth
properties.  Minerbe \cite{Minerbe:Rigidity-Taub-NUT} has characterised the
multi-Taub-NUT metrics amongst all complete four-dimensional
hyperKähler metrics via their volume growth.

\begin{theorem}
  \label{thm:4d-class}
  If \( M \) is a connected complete hyperKähler four-manifold with an
  effective tri-Hamiltonian circle action, then \( M \) is isometric
  to a Gibbons-Hawking metric, an Anderson-Kronheimer-LeBrun metric or
  a Taub-NUT deformation of one of these.
\end{theorem}

From the form of the hyperKähler moment maps for these metrics,
we note the following consequence.

\begin{corollary}
  For a complete hyperKähler four-manifold~\( M \) with effective
  tri\hyphen Hamiltonian circle action, the hyperKähler moment map
  \( \mu\colon M \to \bR^3 \) is surjective.
  \qed
\end{corollary}

To prove the \lcnamecref{thm:4d-class}, let us start by observing that the
stabiliser~\( \stab(x;S^1) \) of each \( x \in M \) under the circle
action is either trivial or all of~\( S^1 \).
Indeed, \( G = \stab(x;S^1) \) is a compact subgroup of \( S^1 \) so
is either all of \( S^1 \) or is finite cyclic.
If \( G = \bZ/k\bZ \), then \( x \)~is not a fixed
point and so the vector field \( X \) is non-zero at~\( x \).
Equivariantly there is a tubular neighbourhood~\( U \) of the orbit
through~\( x \) of the form \( U = S^1 \times_G X^\bot \).
But \( G \) acts trivially on \( X^\bot = \Span{IX,JX,KX} \),
\( G \) stabilises each point in~\( U \).
In particular, as \( M \) is connected, this implies that \( G \)
stabilises points in principal orbits and so \( G \) acts trivially
on~\( M \). 
As \( S^1 \) acts effectively, this implies that \( G \) is trivial.

Now let \( x \in M^X \) be a fixed-point of~\( X \).
There is a tubular neighbourhood \( U \) of \( x \) of the form
\( U = S^1 \times_{S^1} T_xM \).
As \( S^1 \) acts on \( T_xM \) as a connected subgroup~\( H \) of the 
maximal torus~\( \Un(1) \) in~\( \SU(2) \),
it follows that \( H \) is either trivial or all of~\( \Un(1) \).
In the first case, \( X \) is then zero in an open neighbourhood of
\( x \) in~\( M \); since \( M \) is connected and Ricci-flat it
we get that \( X \) is identically zero on~\( M \), which contradicts
the effectiveness of the action.  
We conclude that \( H = \Un(1)  \) and that \( U \) contains no other
fixed-point.
It follows that \( M^X \) is discrete and so countable.

As in Bielawski \cite[Proposition~4.3]{Bielawski:tri-Hamiltonian},
we now have that the map \( \overline \mu\colon M/S^1 \to \bR^3 \)
induced by~\( \mu \) is a local homeomorphism. 
At points of \( M' \), we have \( d\mu = (IX^\flat,JX^\flat,KX^\flat) \),
so \( \overline\mu \)~is a local diffeomorphism on~\( M'/S^1 \).
Examination of the local form of~\( \mu \) then implies that
\( \overline\mu \) is injective in a neighbourhood of each fixed 
point~\( x \) and a local compactness argument implies it is a
homeomorphism there.
The function \( V \) descends to \( N' = M'/S^1 \).
The metric \( V(\alpha_I^2+\alpha_J^2+\alpha_K^2) \) on~\( N' \) is
locally the pull-back of \( \Phi g_{\bR^3} \),
where the conformal factor \( \Phi \)~is a positive function on an
open subset of~\( \bR^3 \) and is harmonic with respect to the
standard metric~\( g_{\bR^3} \).
Around a fixed point~\( x \), we have \( \Phi \) defined and harmonic
in a punctured neighbourhood \( U_q\setminus\{q\} \) of~\( q = \mu(x) \).
Bôcher's Theorem (see~\cite{Axler-BR:Bocher}) states that
\( \Phi(p) = \phi(p) + a_x/\norm{p-q} \) with
\( a_x \geqslant 0 \)~constant and \( \phi \) harmonic on all
of~\( U_q \). 

Let \( S \subset U_q \) be a small distance sphere centred
on~\( q = \mu(x) \).
As \( \overline\mu \) is conformal,
the \( S^1 \)-invariant set \( \mu^{-1}(S) \) is diffeomorphic to a
distance sphere around~\( x \). 
It follows that \( S^1 \to \mu^{-1}(S) \to S \) is a
the Hopf fibration \( S^1 \to S^3 \to S^2 \).
In particular, the bundle \( \mu^{-1}(S) \to S \) has Chern
class~\( \pm1 \).
On the other hand, this fibration has curvature form
\( d\omega = - \Hodge_3 d\Phi \) and Chern class
\( - \frac{a_x}{2\pi} \int_S \Hodge_3 d(\norm{p-q}^{-1}) \).
For the vector field \( \XH \) on~\( \bH \) we have
\( V(\mathbf q) = \norm{\mathbf q}_{\bH}^{-2} \), for
\( \mathbf q \in \bH \).
But equation~\eqref{eq:mu-H} gives
\( \norm{\muH(\mathbf q)} = \tfrac12\norm{\mathbf q}_{\bH}^2 \),
so \( V(\mathbf q) = \frac12\norm{\muH(\mathbf q)}^{-1} \),
giving \( \Phi(p) = \frac12\norm{p}^{-1} \)
and~\( a_{\mathbf 0} = 1/2 \).
We conclude that \( a_x = 1/2 \) for each fixed point~\( x \).
Let us also note that this also shows the orientation conventions here
fix the Chern class to be~\( +1 \).

As in \cite[Proposition~5.3]{Bielawski:tri-Hamiltonian} we may
adjust \( V \) around each fixed point to get a complete
metric~\( g' = \tilde{V}(\alpha_I^2 + \alpha_J^2 + \alpha_K^2) \)
on~\( N' \) with non-negative scalar curvature.
As \( \overline\mu\colon (N',g') \to (\bR^3,g_{\bR^3}) \) is locally
conformal and \( g_{\bR^3} \) is flat,
results of Schoen and Yau \cite[Theorem VI.3.5]{Schoen-Y:lectures} imply that
\( \overline\mu \colon N' \to \bR^3 \) is injective with boundary of
Newtonian capacity zero.
In the terminology of harmonic function theory, this boundary is a
polar set. 

Writing \( \Omega = \mu(M') \subset \bR^3 \),
injectivity of \( \overline\mu \) on~\( N' \) means that we may now
regard \( V \) as a function on~\( \Omega \). 
It is equal to the conformal factor \( \Phi \) and in particular
is a positive harmonic function.
As such it is described by a Martin
representation~\cite{Martin:minimal,Armitage-G:potential}, which we
now determine.

As the topological boundary \( \partial \Omega \) of~\( \Omega \)
in~\( \bR^3 \) is polar,
it has Hausdorff dimension at most~\( 1 \) and so \( \Omega \) is
dense in~\( \bR^3 \). 
By Armitage and Gardiner \cite[Theorem~9.5.1]{Armitage-G:potential},
following~\cite{Naim:Martin-boundary},
the minimal Martin boundary~\( \Delta \) of~\( \Omega \) is
\( \Delta = \partial \Omega \cup \{\infty\} \)
and the Green's function of \( \Omega \) is that of~\( \bR^3 \),
namely \( G(p,q) = 1/\norm{p-q} \).
Adding a constant to~\( \mu \), we may assume that \( 0 \in \mu(M') \).
The Martin kernel is then \( M(p,q) = \norm q / \norm{p-q} \) and
there is a unique measure~\( d\mu_V \) on~\( \Delta \) such that
\( V \) has Martin representation
\begin{equation*}
  V(p) = \int_{\Delta} M(p,q) \,d\mu_V(q),
\end{equation*}
see \cite[Theorem~8.4.1]{Armitage-G:potential}.

Let \( F = \mu(M^X) \) be the image of the fixed-point set.  
Now \( M^X \) is discrete and 
\( \overline\mu\colon M/S^1 \to \bR^3 \) is a local
homeomorphism that is injective on \( N' \). 
It follows that \( \mu \) is injective on~\( M^X \), 
and that \( F \) is a discrete subset of~\( \bR^3 \) that is
contained in \( \partial \Omega \subset \Delta \).
In particular, \( F \) is a Borel set and there is a positive
harmonic function~\( W \) on \( \Omega \) defined by
\begin{equation*}
  W(p) = \int_{F} M(p,q)\,d\mu_V(q).
\end{equation*}
This satisfies
\begin{equation*}
  \lim_{p\to q} \frac{W(p)}{V(p)}
  =
  \begin{dcases*}
    1,&for \( q \in F \),\\
    0,&for \( q \in \Delta \setminus F \),
  \end{dcases*}
\end{equation*}
see \cite[Corollary~9.4.3]{Armitage-G:potential}.
Now \( W(p) \leqslant V(p) \) for all \( p \in \Omega \),
so \( W \) is finite on~\( \Omega \).
But Bôcher's Theorem combined with the local form of~\( V \) gives us
that 
\begin{equation*}
  W(p) = \frac12 \sum_{q \in F} \frac1{\norm{p-q}},
\end{equation*}
for all \( p \in \Omega \).  In particular, Harnack's principle
implies that this sum converges at each \( p \in \bR^3 \setminus F \)
and \( W \)~gives the potential for a Gibbons-Hawking or an
Anderson-Kronheimer-LeBrun metric. 
In particular, \( F \)~has no accumulation point in~\( \bR^3 \).

Now we have
\begin{equation*}
  \phi(p) = (V-W)(p) = \int_{\Delta\setminus F} M(p,q)\,d\mu_V(q)
\end{equation*}
is a positive harmonic function on~\( \Omega \).
Again \( \phi \) is bounded above by~\( V \),
so is finite on~\( \Omega \).
Suppose \( \partial \Omega \) is strictly larger than~\( F \).
Then, there is a \( p_0 \in \Omega \) with Euclidean distance \(
d(p_0,\partial \Omega) \) strictly smaller than \( d(p_0,F) \).
Let \( q \in \partial \Omega \setminus F \) be a point closet
to~\( p_0 \).  
Then the straight-line segment~\( \gamma\colon [0,1) \to \Omega \),
\( \gamma(t) = (1-t)p_0 + tq \), has \( d(\gamma(t),\Omega) =
\norm{\gamma(t)-q} \) for each \( t\in[0,1) \).
The length of the horizontal lift \( \tilde{\gamma} \) of \( \gamma \)
to~\( M \) is
\begin{equation*}
  \ell(\tilde{\gamma}) = \norm{p_0-q} \int_0^1 \sqrt{V(\gamma(t))}\,dt.
\end{equation*}
But \( W \) is bounded on \( \gamma \) and the integral definition
of~\( \phi \) shows that there are constants \( c_i > 0 \) such that
\begin{equation}
  \label{eq:growth}
  \phi(p) \leqslant c_1 + \frac{c_2}{\dist(p,\partial\Omega\setminus F)},
\end{equation}
for all \( p \in \Omega \).  It follows that \( \tilde{\gamma} \) has
finite length, contradicting the completeness of~\( g \),
cf.~\cite[\S5]{Bielawski:tri-Hamiltonian}. 

We conclude that \( F = \partial \Omega \) and that \( V = W + c \),
for some \( c \geqslant 0 \).
For \( c = 0 \), we get the two classes of metrics above;
for \( c>0 \), we have Taub-NUT deformations of them.
This proves the Theorem.

\section{Modification as a twist}
\label{sec:twist}

Given a hyperKähler manifold \( (M,g) \) and a general modification~\(
\Mgmod \), our aim is to find a hyper-Hermitian metric~\( \widetilde g
\) on~\( M \) that gives the  hyperKähler metric of~\( \Mgmod \) under
a twist construction.

We take \( W = \Mgmod \) and choose our circle bundle to be 
\( P = D_0 \) in the notation of~\eqref{eq:D0}.
For this section we work on the subsets of \( M \) and \( \Mgmod \)
that are the images of the projections from~\( D_0 \).
The principal vector fields are \( Z = X_N \) and
\( X' = Y = X - X_N \).

The space \( D_0 \subset \rho^{-1}(0) \subset M \times N \) is a
Riemannian submanifold and the fibration to~\( \Mgmod \) is a
Riemannian submersion.  
So there is a natural horizontal distribution
\begin{equation*}
  \hor = Y^\bot \cap TD_0
\end{equation*}
which we will use in the twist.
The combination of \( \hor \) and the principal vector fields
determines the corresponding connection one forms. 

Consider the principal fibration \( \pi\colon D_0 \to M \).
Write \( i\colon D_0 \to M\times \bH \) for the inclusion.
The connection one-form \( \theta \) is required to satisfy
\( \theta(X_N) = 1 \) and \( \ker\theta = \hor \).  
Equivariance then follows from the equivariance of~\( \hor \).  

The first condition \( \theta(X_N) = 1 \) is satisfied by the form
\begin{equation*}
  \theta'= i^*(X_N^\flat/\norm{X_N}^2) = i^*(V_NX_N^\flat),
\end{equation*}
where we have written the hyperKähler metric \( g_N \) on~\( N \) as
in~\cref{sec:hyperk-four-manif}, so \( V_N \)~is a function
of~\( \mu_N \).
On \( \rho^{-1}(0) \), we have \( \mu = \mu_N \), so we may regard
\( i^*V_N \) as a function \( V_N(\mu) \) of~\( \mu \).

We now need to adjust \( \theta' \) to get a one-form vanishing
on~\( \hor \).  As 
\begin{equation*}
  d\rho^A = Y \hook (\omega_A + \omega^N_A)
  = X \hook g(A\cdot,\cdot)
  - X_N \hook g_N(A\cdot,\cdot) = (AX-AX_N)^\flat,
\end{equation*}
for \( A = I,J,K \), the tangent bundle to \( TD_0 \) is
\begin{equation*}
  TD_0 = \ker d\rho =
  (IX-IX_N)^\bot \cap (JX-JX_N)^\bot \cap (KX - KX_N)^\bot.
\end{equation*}

The form \( \theta' \) is zero on all vectors in~\( \hor \) apart from
those with a component along~\( X_N \).
Any vector in \( \hor = \Span Y^\bot = \Span{X-X_N}^\bot \) may be
decomposed as a vector orthogonal to both \( X \)
and~\( X_N \) and a component proportional to an appropriate
linear combination \( X_\ell = eX+fX_N \).
The condition \( \inp{X_\ell}Y=0 \) gives a standard choice with
\( X_\ell = \norm{X_N}^2X+\norm X^2X_N =  V_N^{-1} X+\norm X^2 X_N \).

Put \( \theta = \theta' + \lambda X^\flat \).
Then \( \theta(X_N) = 1 \), as required.
We have \( \theta(X_\ell) = \theta'(X_\ell) + \lambda X^\flat(X_\ell)
= \norm X^2 + \lambda V_N^{-1} \norm X^2 \).
Therefore
\begin{equation*}
  \theta = V_N\, i^*(X_N^\flat - X^\flat).
\end{equation*}

\begin{lemma}
  The connection \( \theta \) on \( D_0 \to M \) has curvature 
  \begin{equation}
    \label{eq:F}
    F = - \bigl(\Hodge_3 dV_N + d(V_N \alpha_0)\bigr)
  \end{equation}
  where \( \Hodge_3 \) is the Hodge \( \Hodge \)-operator on the
  three-dimensional distribution with orthonormal basis \( IX,
  JX, KX \) and \( \alpha_0 = X^\flat \in \Omega^1(M) \).
\end{lemma}

\begin{proof}
  We have \( \pi_1^*F = d\theta \) and
  \begin{equation*}
    d\theta = i^* d\bigl(V_N(\alpha_0^N - X^\flat)\bigr)
    = i^*d(\beta_0 - V_NX^\flat)
    = - i^*\bigl(\Hodge_3 dV_N + d(V_N X^\flat)\bigr).
  \end{equation*}
  Now \( i^*V_N \) is a function of \( \mu = (\mu^I,\mu^J,\mu^K) \), so
  \( i^*dV_N = V_1d\mu^I + V_2d\mu^J + V_3d\mu^K = V_1(IX)^\flat +
  V_2(JX)^\flat + V_3(KX)^\flat \) and the result follows.
\end{proof}

We see that
\begin{equation*}
  X \hook F = X \hook d\theta = L_X\theta - d(X\hook\theta) =
  d(V_N\norm X^2).
\end{equation*}
Thus \( a= k - V_N\norm X^2 \) for some constant \( k \).
This constant is determined by
\begin{equation*}
  X - X_N = Y = \widetilde X + aX_N
  = (X+V_N\norm X^2 X_N) + aX_N = X + kX_N.
\end{equation*}
So \( k=-1 \) and
\begin{equation}
  \label{eq:a}
  a = -(1+V_N\norm X^2).
\end{equation}

\begin{lemma}
  The pull-backs of the original hyperKähler metric \(g \) on~\( M \)
  and of the hyperKähler metric \( \gmod \) on the
  general modification are related by
  \begin{equation*}
    \pi_1^*g-\pi^*\gmod = - V_N\, \pi_1^*\gHX\quad\text{on \( \hor \)}, 
  \end{equation*}
  where \( \gHX = \alpha_0^2 + \alpha_I^2 + \alpha_J^2 + \alpha_K^2
  \) with \( \alpha_I = IX^\flat = X \hook \omega_I \), etc.
\end{lemma}

\begin{proof}
  The fibration \( D_0 \to \Mgmod \) is a Riemannian submersion with \(
  \hor \) orthogonal to the fibres, so \( \pi^* \gmod = g_D|_\hor \),
  with \( g_D = i^*g + i^*g_N \) the metric induced from the product
  \( M \times N \).
  For \( A \in \hor \), we have \( g_D(A,Y) = 0 \),
  for \( Y = X-X_N \), so \( g_D(A,X_N) = g_D(A,X) \),
  and we may write \( A = A_1 + B \),
  with \( A_1 \) orthogonal to the quaternionic span of
  \( X_\ell = V_N^{-1} X + \norm X^2X_N \) and
  \( B = s_AX_\ell - s_{IA}IX_\ell - s_{JA}JX_\ell - s_{KA}KX_\ell \),
  where \( s_A = g_D(A,X_\ell)/\norm{X_\ell}^2 = V_N g_D(A,X)/\norm X^2 \).
  Then the projection of \( A \in TD_0 \subset TM \times TN \)
  to~\(TM \) is \( A_2 = A_1 + V_N^{-1} \tilde B \),
  with \( \tilde B = s_AX - s_{IA}IX - s_{JA}JX - s_{KA}KX \).
  We have
  \begin{equation*}
    \begin{split}
      \MoveEqLeft (\pi_1^*g - \pi^*\gmod)(A,A)\\
      &= g_D(A_2,A_2) - g_D(A,A) 
      = g_D(A,A_2) - g_D(A,A) \\
      &= g_D\bigl(A,\norm X^2 (-s_AX_N + s_{IA}IX_N + s_{JA}JX_N
        + s_{KA}KX_N)\bigr)  \\
      &= - V_N \bigl(g_D(A,X)^2 + g_D(A,IX)^2 + g_D(A,JX)^2
        + g_D(A,KX)^2\bigr) \\
     &= - V_N\pi_1^*g_N(A,X)^2,
    \end{split}
  \end{equation*}
  as claimed.
\end{proof}

We conclude

\begin{theorem}
  \label{thm:twist-as-mod}
  The general modification \( (M,g,X) \) by \( N^4 \)
  with hyperKähler metric \( g_N = \dfrac1{V_N}(dt+\omega^2) + V_N
  g_{\bR^3} \), is the twist of
  \begin{equation*}
    \tilde g = g + V_N(\mu)\, \gHX
  \end{equation*}
  by \( X \) with respect to
  \( F = - \bigl(\Hodge_3 dV_N + d(V_N \alpha_0)\bigr) \) and
  \( a = -(1+V_N\norm X^2) \).
  \qed 
\end{theorem}

It follows that an ordinary modification has \( V_N(\mu) =
V_{\bH}(\mu) = 1/(2\norm{\mu}) \).  One may check directly
that the above formulae give hyperKähler metrics.  However, in the
next section we will consider a more general situation, and determine
when the twist construction leads to hyperKähler structures.

Note that when \( \dim M = 4 \),
with \( g \) given by a potential~\( V \), as in \eqref{eq:G-H},
then the general modification by~\( N^4 \) has
potential~\( V + V_N \).
Taub-NUT deformations correspond to \( V_N = c > 0 \) constant,
which is the potential for the flat metric on \( S^1 \times \bR^3 \),
where \( S^1 = \bR/2\pi r\bZ \) with \( c = 1/r^2 \).

\section{General hyperKähler twists}
\label{sec:gener-hyperk-twists}

Let us start with a connected hyperKähler manifold \( (M,g,I,J,K) \)
together with a tri-holomorphic isometry generated by \( X \).
This action will be taken to be effective,
but not necessarily free.
We assume that \( \dim M > 4 \).
As above, write
\begin{gather*}
  \alpha_0 = X^\flat,\quad \alpha_I = IX^\flat = X^\flat \circ (-I),
  \quad \alpha_J = JX^\flat,\quad \alpha_K = KX^\flat\\
  \gHX = \alpha_0^2 + \alpha_I^2 + \alpha_J^2 + \alpha_K^2.
\end{gather*}
Consider an \emph{elementary quaternionic deformation}~\(\tilde g \)
of the metric~\( g \) given by
\begin{equation}
  \label{eq:elem-def}
  \tilde g = f\,g + h\,\gHX,
\end{equation}
where \( f \) and \( h \) are smooth functions.
Note that when \( \tilde g \) is non-degenerate,
it gives a, possibly indefinite, hyper-Hermitian metric compatible
with \( I,J,K \).
We wish to determine which twists~\( W \) of \( (M,\tilde g,I,J,K) \)
are hyperKähler when we use the symmetry~\( X \) together with
arbitrary twist data~\( (F,a) \).

The Kähler form \( \omega^W_I \) of the induced structure on~\( W \)
is by definition \Hrelated to
\begin{equation*}
  \tilde \omega_I = \tilde g(I\cdot,\cdot)
  = f\,\omega_I + h(\alpha_{0I}+\alpha_{JK}),
\end{equation*}
where \( \alpha_{0I} = \alpha_0 \wedge \alpha_I \), etc.
By \eqref{eq:dW},
the exterior derivative of \( \omega^W_I \) is then \Hrelated
to \( d_W\tilde \omega_I \), where
\begin{equation*}
  d_W = d - \frac1a F \wedge (X \hook \cdot).
\end{equation*}
Thus in order for \( W \) to be hyperKähler we need
\begin{equation*}
  d_W\tilde\omega_I = 0 = d_W\tilde\omega_J = d_W\tilde\omega_K.
\end{equation*}
Noting that \( d\alpha_I = d(X\hook\omega_I) = L_X\omega_I = 0 \),
we compute the first of these to be
\begin{equation}
  \label{eq:dW-tomI}
  \begin{split}
    d_W\tilde\omega_I
    &= df \wedge \omega_I + d(h\alpha_0)\wedge\alpha_I
    + dh\wedge\alpha_{JK} \eqbreak
    - \frac1a F \wedge f\alpha_I
    - \frac1a F h\norm X^2\wedge \alpha_I\\
    &= df \wedge \omega_I + H\wedge\alpha_I + dh\wedge\alpha_{JK},
  \end{split}
\end{equation}
where
\begin{equation*}
  H = d(h\alpha_0) - \frac1a (f+h\norm X^2)F.
\end{equation*}

Consider the orthogonal decompositions \( TM = \bH X + V \) and
\( T^*M = \bH\alpha_0 + V^* \).
At fixed points of \( X \), the first summand is \( \{0\} \).
We declare elements of \( \bH\alpha_0 \) to be of type~\( (1,0) \),
and those of \( V^* \) to be of type~\( (0,1) \).
This then gives a type decomposition of the exterior
algebra~\( \Omega^*(M) \).
For example we have \( \omega_I = \omega_I^{2,0} + \omega_I^{0,2} \).
Splitting the equation \( d_W\tilde\omega_I = 0 \) in to components
gives four equations, the first two of which are:
\begin{alignat}{2}
  \label{eq:03}
  \text{Type \( (0,3) \):}&\quad& df^{0,1} \wedge \omega_I^{0,2} &= 0,\\
  \label{eq:12}
  \text{Type \( (1,2) \):}&& df^{1,0} \wedge \omega_I^{0,2} +
  H^{0,2}\wedge\alpha_I &=0.
\end{alignat}
\Cref{eq:03} implies that \( df^{0,1} = 0 \).
Wedging \cref{eq:12} with \( \alpha_I \),
then gives \( df^{1,0}\wedge\alpha_I=0 \).
Using the corresponding equations from
\( d_W\tilde\omega_J = 0 = d_W\tilde\omega_K \),
we conclude that \( df = 0 \).
Thus \( f \) is constant, and up to a homothety we may take \( f = 1 \).

With \( f = 1 \), the equation \( d_W\tilde\omega_I = 0 \) reduces to
\begin{equation}
  \label{eq:f1rest}
  H\wedge\alpha_I + dh\wedge\alpha_{JK} = 0
\end{equation}
Again wedging with \( \alpha_I \) gives \( dh\wedge\alpha_{IJK} = 0 \)
and so
\begin{equation*}
  dh = h_I\alpha_I + h_J\alpha_J + h_K\alpha_K
\end{equation*}
on the set \( M' = M \setminus M^X \) where \( X \ne 0 \),
for some functions \( h_I \), \( h_J \), \( h_K \).
Now \cref{eq:f1rest} is \( (H+h_I\alpha_{JK})\wedge\alpha_I = 0 \).
It follows that on \( M' \) we have
\begin{equation*}
  H = - h_I\alpha_{JK} - h_J\alpha_{KI} - h_K\alpha_{IJ}
  = - \Hodge_3 dh,
\end{equation*}
where \( \Hodge_3 \) is the Hodge star operator of \( \gHX \) on
\( \Span{\alpha_I,\alpha_J,\alpha_K} \).

On \( M' \), the definition of \( H \) implies now that
\begin{equation}
  \label{eq:aFhX}
  \frac1a F(1+h\norm X^2) = d(h\alpha_0) + \Hodge_3 dh.
\end{equation}
Contracting \eqref{eq:aFhX} with \( X \) gives
\begin{equation*}
  - (d\log a)(1+h\norm X^2)
  = X \hook d(h\alpha_0)
  = L_X(h\alpha_0) - d(X\hook h\alpha_0)
  = - d(h\norm X^2),
\end{equation*}
so \( d\log a = d \log(1+h\norm X^2) \) and we conclude that
\begin{equation}
  \label{eq:aa}
  a = \lambda(1+h\norm X^2)
\end{equation}
for some (non-zero) constant~\( \lambda \).
Returning to \cref{eq:aFhX}, we impose \( dF = 0 \), to get
\begin{equation*}
  0 = \frac 1\lambda dF = d\Hodge_3dh.
\end{equation*}
Summarising, we have proved the following result.

\begin{theorem}
  \label{thm:hK-twist}
  Let \( (M,g,I,J,K) \) be a connected hyperKähler manifold with
  an effective tri\hyphen holomorphic isometry~\( X \).
  Up to homothety and topological considerations,
  an elementary quaternionic deformation~\( \tilde g \)
  has a twist with respect to \( (X,a,F) \) that is pseudo-hyperKähler
  if and only if on the dense open set \( M' = M \setminus M^X \)
  we have
  \begin{equation*}
    \tilde g = g + h\,\gHX,\quad
    a = \lambda(1+h\norm X^2)\ne 0,\quad
    F = \lambda(d(h\alpha_0) + \Hodge_3dh)
  \end{equation*}
  for some non-zero constant \( \lambda \) and some function \( h \)
  with \( dh \in \Span{\alpha_I,\alpha_J,\alpha_K} \) harmonic:
  \( d\Hodge_3dh = 0 \).

  If \( X \) is tri-Hamiltonian, then \( \alpha_I = d\mu^I \), etc.,
  and \( h \) is locally a pull-back of a harmonic function
  \( h(\mu_I,\mu_J,\mu_K) \) in~\( \bR^3 \).
  \qed 
\end{theorem}

As \( \tilde{g}(X,X) = (1+h\norm X^2)\norm X^2 \),
the condition \( a \ne 0 \) guarantees that \( \tilde{g} \)
is non-degenerate.
We get a positive definite hyperKähler metric
if \( 1 + h\norm X^2 > 0 \).

The \enquote{topological considerations} above are
\begin{enumerate*}
\item that \( F \) has integral periods,
  so a twist bundle~\( P \) exists,
  and
\item that the resulting lifted action on~\( P \)
  generated by \( X' = \widetilde X + aZ \)
  gives a smooth quotient \( W = P/\Span{X'} \).
\end{enumerate*}
The freedom in the choice of \( \lambda \) can help in achieving
these conditions.
Otherwise \( \lambda \) is often irrelevant as \( a \) and
\( F \) usually occur in the combination \( \tfrac1aF \).

\begin{proposition}
  \label{prop:dual}
  Suppose as in \cref{thm:hK-twist} that an elementary deformation of
  \( (M,g,X) \) twists via \( (h,\lambda,a,F) \) to
  \( (\check M,\check g \Hrel \tilde g,\check X \Hrel -\frac1aX) \).
  Then \( (M,g,X) \) is obtained from a twist of an elementary
  deformation of \( (\check M,\check g,\check X) \) by the data
  \begin{equation*}
    \check h \Hrel - \lambda^2 h,\quad
    \check \lambda \Hrel \frac 1\lambda,\quad
    \check a \Hrel \frac1a,\quad
    \check F \Hrel \frac1aF.
  \end{equation*}
\end{proposition}

\begin{proof}
  In general if \( W \) is the twist of \( M \) by \( (X,a,F) \) then
  the construction is inverted by the objects \Hrelated to \(
  (-\tfrac1a X,\tfrac1a,\tfrac1aF) \), see \cite[\S3.2]{Swann:twist}.

  In \cref{thm:hK-twist}, \( a = \lambda(1 + h\norm X^2) \).
  The construction is inverted with \( \check X \Hrel -\frac1aX \),
  \( \check a \Hrel 1/a \), \( \check g \Hrel \tilde g = g + h\gHX \) and
  \( \tilde{\check g} = \check g + \check h\gHXc \Hrel g \).
  We then have \( \check \alpha_0 = \check X \hook \check g
  \Hrel -\tfrac1a X \hook \tilde g = -\tfrac1\lambda\alpha_0 \) and so
  \( \norm{\check X}^2 = \check X \hook \check \alpha_0
  \Hrel \tfrac 1{\lambda a}X \hook \alpha_0
  = \tfrac 1{\lambda a}\norm X^2 \).
  Taking \( \check \lambda = 1/\lambda \), it follows that
  \begin{equation}
    \label{eq:inv-h}
    \check h = (\tfrac{\check a}{\check \lambda}-1)/\norm{\check X}^2
    \Hrel (\tfrac\lambda a-1)/(\tfrac1{\lambda a}\norm X^2)
    = \lambda^2 (1-\frac a\lambda)/\norm X^2 = -\lambda^2 h.
  \end{equation}
  As \( \check\alpha_I \Hrel -\tfrac1\lambda\alpha_I \),
  we have \( \check F = \check \lambda(d(\check h\check \alpha_0)
  + \Hodge_3d\check h) 
  \Hrel  d(h\alpha_0) - \tfrac1a h\norm X^2 F + \Hodge_3 dh
  = \tfrac 1aF \), showing that
  the curvature forms are correctly related. 
\end{proof}

By \cref{thm:twist-as-mod}, the general hyperKähler modification
corresponds to choosing \( \lambda = -1 \)
and \( h = V_N(\mu) \).

\begin{corollary}
  \label{cor:unmodification}
  Suppose \( (M,g,I,J,K) \) is a hyperKähler manifold
  with a tri\hyphen Hamiltonian isometry~\( X \).
  Then the hyperKähler twist in \cref{thm:hK-twist} of
  \( \tilde g = g + h\gHX \) with \( h = -V_N(\mu) \)
  inverts the general modification construction.
\end{corollary}

\begin{proof}
  Suppose that \( M \) of the \namecref{cor:unmodification} is
  the general modification of the hyperKähler manifold
  \( (\check M,\check g,\check I,\check J,\check K) \)
  by~\( (\check X,\check \mu,V_N(\check \mu)) \).
  Then taking \( \check \lambda = -1 \), we have
  \( \check h = V_N(\check \mu) \).
  Furthermore \( d\check \mu^I = \alpha_{\check I}
  \Hrel \alpha_I = d\mu^I \),
  so we may take the function~\( \mu \) that is \Hrelated to
  \( \check \mu \) as the hyperKähler moment map.
  This gives \( V_N(\check \mu) \Hrel V_N(\mu) \) and it follows
  that \( h \Hrel -\check h \) is \( h = -V_N(\mu) \),
  as claimed. 
\end{proof}

Rereading the proofs of \cref{sec:twist},
one sees that on the smooth set this inversion may be interpreted
as a general modification with respect to \( N \) equipped with
the negative definite hyperKähler metric with potential \( -V_N \).

In the case of an ordinary modification, the potential function gives
\( h = 1/(2\norm \mu) \).
For the above inverse modification to produce a non-degenerate
positive definite hyperKähler metric we need \( 1 + h\norm X^2 > 0 \),
i.e.,
\begin{equation}
  \label{eq:unmod}
  \norm X^2 < 2\norm\mu,
\end{equation}
away from zeros of~\( \mu \).
On \( \bR^4 = \bH \), with the standard circle action
and \( \mu = \muH \), we have  that
\( \norm \XH^2 = 1/V_\bH = 2\norm\mu \).

\begin{corollary}
  The standard flat hyperKähler structure on \( \bR^4 \) is not
  the hyperKähler modification of any other hyperKähler manifold.
\end{corollary}

\begin{proof}
  Any tri-holomorphic isometric circle action on \( \bH = \bR^4 \)
  is a subgroup of \( \Sp(1) \) and so is conjugate to the standard
  action generated by~\( \XH \).
  In general \( \mu = \muH + c \) for some \( c \in \im\bH \).
  However, the image of \( \muH \) is all of \( \im\bH \),
  so there is a \( p \in \mu^{-1}(0) \).
  The limit of \( (2\norm\mu-\norm \XH^2)(q) \) as \( q\to p \)
  is \( -\norm \XH^2(p) \).
  If \( \bR^4 \) is a hyperKähler modification,
  this limit is non-negative,
  so we have \( \norm\XH^2 = 0 \) at~\( p \).
  This implies \( p = 0 \), and \( \mu = \muH \),
  but now \eqref{eq:unmod} is satisfied nowhere,
  and \( \bR^4 \) can not be a modification of any other
  hyperKähler manifold.
\end{proof}

Note that topologically \( \bR^4 \) does carry metrics that are
hyperKähler modifications;
these are given by Taub-NUT deformations of the flat metric.
The potential is \( V = 1/(2\norm\mu) + c \), with \( c > 0 \);
these are the original Taub-NUT metrics.
By the discussion at the end of \cref{sec:twist},
these may also be regarded as a standard modification of
\( S^1 \times \bR^3 \).
Just like the flat metric,
the circle action on the Taub-NUT metric has a unique fixed-point.

\section{Pseudo-Riemannian structures}
\label{sec:pseudo-riem-struct}

Up to this point we have worked with hyperKähler geometries where the
metric is positive definite.  However, much of what we have discussed
applies also to pseudo-Riemannian structures.

\begin{theorem}
  Let \( (M,g,I,J,K) \) be a pseudo-hyperKähler manifold with a
  locally free, tri\hyphen holomorphic isometry~\( X \).
  Then the only twists that are pseudo-hyperKähler are given by the
  data of \cref{thm:hK-twist}.  
\end{theorem}

\begin{proof}
  At points where \( X \) is not null,
  the proof above goes through without change. 
  If \( X \) is null is null at~\( p \), but non-zero,
  choose a local vector field~\( Y \) such that
  \( g(X,Y) = 1 \) and
  \( g(AX,Y) = 0 \) for \( A=I \), \( J \) and~\( K \).
  Then \( TM = \bH X + \bH Y + V' \),
  where \( V' \) is the orthogonal complement to \( \bH X + \bH Y \).
  Writing \( \beta_0 = Y^\flat \), etc., we have
  \begin{equation*}
    \omega_I = \alpha_0\beta_I - \alpha_I\beta_0 + \alpha_J\beta_K -
    \alpha_K\beta_J + \omega_I^{0;2},
  \end{equation*}
  where we have declared elements of \( \bH \alpha_0 + \bH \beta_0 \) to
  be type~\( (1;0) \) and those of \( (V')^* \) to be
  type~\( (0;1) \).
  We now consider the equation \( d_W\tilde\omega_I = 0 \),
  with \( d_W\tilde\omega_I \) given by~\eqref{eq:dW-tomI}.
  The \( (0;3) \)-component says \( df^{0;1} \wedge \omega_I^{0;2} = 0 \),
  so \(df^{0;1} = 0 \).
  The \( (1;2) \)-component gives
  \( df^{1;0}\wedge \omega_I^{0;2} + H^{0;2}\wedge\alpha_I = 0 \),
  which implies \( df^{1;0}\wedge\alpha_I = 0 \).
  Imposing \( d_W\tilde\omega_J=0=d_W\tilde\omega_K \),
  we conclude that \( df = 0 \), and again we may take \( f \equiv 1 \).
  The rest of the proof now proceeds as in the definite case.
\end{proof}

\section{Singular behaviour}
\label{sec:singular-behaviour}

Returning to the Riemannian setting, we wish to analyse some of the
local singular behaviour and apply this to complete hyperKähler
metrics. 
To be more precise, suppose we have two hyperKähler manifolds \( M \)
and \( \check M \) which have open dense sets related to each other by
a combination of twist and elementary deformation as in
\cref{thm:hK-twist}.
We consider what singular behaviour is allowed for the deformation
function~\( h \) and what can happen at fixed points of the
symmetries. 

We start by analysing the moment map~\( \mu \) around a zero
of~\( X \).

\begin{proposition}
  \label{prop:open}
  Suppose \( M \) is a hyperKähler manifold with an effective
  tri\hyphen Hamiltonian action of a circle.
  Then the hyperKähler moment map \( \mu\colon M \to \bR^3 \) is an
  open map.
\end{proposition}

\begin{proof}
  Away from the fixed point set \( M^X \), the derivative \( d\mu \)
  has full rank, so \( \mu \) is open on \( M' = M\setminus M^X \).
  We need to describe the behaviour around fixed points.

  Fix \( p \in M^X \).  
  Then \( S^1 \) acts on \( T_pM \) as a non-trivial compact subgroup
  of~\( \Sp(n) \) with infinitesimal generator~\( \nabla X_p \).
  It therefore acts a subgroup of a maximal torus and
  there is an orthogonal quaternionic splitting \( T_pM = W \oplus Z \),
  where \( Z \) is a trivial module and \( W \) is a direct sum of
  non-trivial irreducible \( S^1 \)-modules.
  The trivial module \( Z \) is the tangent space to \( M^X \) at~\( p \).
  Let \( W_1 \leqslant W \) be an \( S^1 \)-invariant quaternionic
  module of real dimension~\( 4 \); write \( W_2 \) for its orthogonal
  complement in~\( W \).
  Then \( T_pM = W_1 \oplus W_2 \oplus Z \),
  orthogonally and quaternionically.
  
  The exponential map at \( p \) is a local diffeomorphism onto an open
  normal neighbourhood \( U \) of~\( p \).  We may write
  \( M_1 = \exp_p(W_1) \cap U \), which is an \( S^1 \)-invariant
  submanifold.  Let \( (x^1,\dots,x^{4n}) \) be Riemannian normal
  coordinates on~\( U \), with \( \partial/\partial x^i|_p \), for
  \( i=1,\dots,4 \), an orthonormal quaternionic basis for~\( W_1 \).
  Then \( (x^1,\dots,x^4) \) are coordinates on~\( M_1 \), identifying
  \( M_1 \) with a neighbourhood of the origin in~\( \bH \), in such a
  way that the quaternionic structures agree at~\( p \).  Moreover this
  identification is equivariant for some non-zero multiple~\( k \)
  of~\( \XH \) on~\( \bH \).  Let \( \mu_0 \) be the pull-back
  to~\( M_1 \) of the hyperKähler moment map~\( k\muH \) on~\( \bH \)
  for~\( k\XH \).
  
  On \( M^X \), we have
  \( d\mu = X \hook (\omega_I,\omega_J,\omega_K) = 0 \), so \( \mu \) is
  locally constant on~\( M^X \).  We may assume for convenience that
  \( \mu(p) = 0 \).  On the other hand \( X \) only vanishes at~\( p \)
  on \( M_1 \), so the map \( \overline\mu\colon M_1/S^1 \to \bR^3 \)
  induced by~\( \mu \) is a local diffeomorphism away from~\( p \).  The
  second derivative of~\( \mu \) at~\( p \) is determined by
  \( \nabla X_p \) via
  \begin{equation*}
    \nabla^2\mu_I = \nabla (d\mu_I) = \nabla (IX)^\flat = I \nabla X^\flat.
  \end{equation*}
  etc.  It follows that \( \mu \) and \( \mu_0 \) agree to second order
  at \( p\in M_1 \).

  The moment map \( \muH \) induces a homeomorphism
  \( \overline\muH \colon \bH/\Span{\XH} \to \bR^3 \) and satisfies
  \( \muH(B_{\bH}(0,r)) = B_{\bR^3}(0,r/2) \) for each \( r > 0 \).
  This implies \( \mu_0(B(p,r)) = B_{\bR^3}(0,\abs k r/2) \) for small~\( r \).
  We may now write \( \mu - \mu_0 = \Psi = \Psi_0 \circ \mu_0 \)
  on~\( M_1 \) and note that \( d\Psi_p = 0 \).  
  There is an \( r > 0 \) such that
  \( \norm{d\Psi} < 1/(4\abs k) \) on
  \( \overline B_{M_1}(p,r) \subset U \),
  so 
  \( \Psi(\overline B_{M_1}(p,r)) \subset \overline
  B_{\bR^3}(0,r/(4\abs k)) \).   
  It follows that
  \( \Psi_0(\overline B_{\bR^3}(0,r/2)) \subset \overline
  B_{\bR^3}(0,r/4) \).

  We claim that \( \overline \mu \) is a homeomorphism
  \( \overline B_{M_1}(p,r)/S^1 \to \overline B_{\bR^3}(0,r/2) \).
  For \( v \in \overline B_{\bR^3}(0,r/4) \), consider
  \( k_v\colon \overline B_{\bR^3}(0,r/2) \to \overline
  B_{\bR^3}(0,r/2) \), defined by \( k_v(u) = v - \Psi_0(u) \).
  This is a metric space contraction with factor \( 1/2 \),
  so by Banach's Fixed\hyphen Point Theorem there is a
  unique~\( u = \overline\mu_0(x) \) satisfying
  \( \overline\mu(x) = v \).
  Thus \( \overline\mu \) is a continuous bijection, and by
  compactness, a homeomorphism.

  It follows that \( \mu(p) \) is an interior point
  of~\( \mu(U) \subset \bR^3 \) and we have proved that \( \mu \) is
  an open map.
\end{proof}

Note that the corresponding result is not true for hyperKähler moment
maps for non-Abelian group actions, see \cite{Dancer-S:cuts}.

\begin{corollary}
  \label{cor:level-dim}
  Let \( L \) be a component of a level set of a hyperKähler moment
  map \( \mu \colon M \to \bR^3 \) for an effective circle action.
  Then either
  \begin{enumerate}
  \item \( L \) has codimension~\( 4 \) and is smooth, or 
  \item \( L \) has codimension~\( 3 \) and any singular set is of
    codimension at least~\( 5 \) in~\( L \).
  \end{enumerate}
\end{corollary}

\begin{proof}
  Let \( p \) be a point of~\( L \).
  If \( X_p \ne 0 \), then \( \mu \) is regular at~\( p \).
  It follows that and near~\(p \), the set \( L \) is smooth and of
  codimension~\( 3 \). 

  If \( X_p = 0 \), then \( p \in M^X \).
  Adding a constant to \( \mu \), we may take \( \mu(p) = 0 \).
  Write \( C \) for the component of \( M^X \) that contains~\( p \).
  Then \( C \)~is smooth of codimension~\( 4\ell \) for some
  \( \ell \geqslant 1 \) and \( C \subset L \).
  Using the notation of the previous proof,
  fix an \( S^1 \)-invariant four-dimensional slice
  \( M_1 = \exp_p(W_1) \cap U\) that is equivariantly diffeomorphic to a
  neighbourhood of \( 0 \in \bH \) with circle action generated
  by~\( k\XH \).
  Then we saw that the moment map~\( \mu \) induces a
  homeomorphism~\( \overline\mu \) from \( M_1/\Span X \) to a
  neighbourhood of the origin in~\( \bR^3 \).

  If \( \codim C = 4 \), then \( W_1 \)~is the whole fibre~\( W \) of
  the normal bundle of~\( C \).
  As \( \overline\mu \) is a homeomorphism, we get that \( C = L \).
  Thus \( L \) is  smooth of codimension~\( 4 \). 
  
  If \( \codim C = 4\ell > 4 \),
  then we claim that \( L \) contains a regular point of~\( \mu \).
  Fix \( \varepsilon > 0 \) sufficiently small, so that each \( \delta
  \in (0,\varepsilon]\)
  the set \( M_\delta = \exp_p(W)\cap B_\delta(p) \) be a
  transverse slice to~\( C \)
  in the \( \delta \)-ball around~\( p \).
  Suppose \( M_\delta \setminus\{p\} \) contains no zero of~\( \mu \),
  so \( \mu (M_\delta \setminus\{p\}) \subset \bR^3\setminus\{0\} \).
  It follows that we get a continuous map \( S^{4\ell - 1} \to S^2 \).
  Now the restriction of this map to \( S^3 \subset M_1 \cap B_\delta(p) \)
  is a non-trivial circle bundle, but \( S^3 \) is homotopic to a point in
  \( S^{4\ell - 1} \), \( \ell \geqslant 2 \), so this is impossible.
  It follows that \( M_\delta\setminus\{p\} \)
  meets~\( \mu^{-1}(0) \) for each \( 0 < \delta \leqslant \varepsilon \).
  In particular \( p \) lies in the closure of
  \( Z = (\mu^{-1}(0)\cap M_\varepsilon)\setminus\{p\} \). 
 
  Now \( (M,g) \) is Ricci-flat, so the geometry and the vector
  field~\( X \) are real analytic in harmonic coordinates.
  It follows that \( \mu \) is real analytic, so \( Z \) is
  a semi-analytic set.
  The Curve Selection Lemma for semi-analytic sets
  \cite{Wall:regular,Milnor:singular},
  or the Rectilinearization Theorem
  \cite[Theorem~0.2]{Bierstone-M:semianalytic},
  implies that there is a \( C^1 \) (or analytic) curve
  \( \gamma\colon [0,1] \to M_\varepsilon \), with \( \gamma(0) = p \)
  and \( \gamma(0,1] \subset Z \).
  Thus the component~\( L \) of \( \mu^{-1}(0) \) containing~\( p \)
  contains a regular point~\( q = \gamma(1) \)
  of~\( Z \subset \mu^{-1}(0) \).
  It follows that \( \codim L = 3 \) and the singular points
  of~\( L \) lie in fixed points sets for~\( X \) of codimension at
  least~\( 8 \) in~\( M \), so at least codimension~~\( 5 \)
  in~\( L \).
\end{proof}

\begin{lemma} 
  \label{lem:estimate}
  Suppose \( M \) is a hyperKähler manifold and that \( X \) generates
  a tri-Hamiltonian circle action with hyperKähler moment map~\( \mu \).
  If \( p \) is a fixed point of~\( X \),
  then there is an open neighbourhood \( U \) of \( p \) and a
  constant~\( c > 0 \) such that
  \( c\norm{X_q}^2 \geqslant \norm {\mu(q) - \mu(p)} \)
  for all \( q \in U \).
\end{lemma}

\begin{proof}
  Choose geodesically convex neighbourhoods~\( V,W \) of~\( p \)
  in~\( M \) with \( M^X \cap \overline V \) connected, \( \overline V \)
  compact and \( \overline V \subset W \).
  Note that \( d\mu = X\hook (\omega_I,\omega_J,\omega_K) \) is zero
  at points of~\( M^X \),
  so \( \mu \) is constant on \( M^X \cap \overline V \).

  For \( q \in W \), let \( p' \in M^X \cap \overline V \) be the
  closest point to~\( q \) and let \( \gamma \) be the minimal unit
  speed geodesic in~\( W \) from \( p' = \gamma(0) \) to~\( q = \gamma(T) \).
  Then for \( t \in [0,T] \) we have
  \begin{equation}
    \label{eq:rough-estimate}
    \norm{\mu(\gamma(t)) - \mu(0)}
      \leqslant t \sup_{s \in [0,t]} \norm{d(\mu\circ \gamma)_{\gamma(s)}}
  \end{equation}
  As \( d\mu = (IX^\flat,JX^\flat,KX^\flat) \), we have
  \( \norm{d\mu}^2 = 3\norm X^2 \).
  As \( \gamma \) is unit speed we thus have
  \( \norm{d(\mu\circ\gamma)_{\gamma(s)}}
  \leqslant \sqrt3 \norm{X_{\gamma(s)}} \).

  Note that \( v = \dot\gamma(0) \) is a unit vector orthogonal to
  \( T_{p'}M^X = \ker \nabla X \).
  Let \( f(t;v,p') = f(t) = \norm{X_{\gamma(t)}}^2 \).
  Then \( f(0) = 0 = \dot f(0) \) and
  \( \ddot f(0) = 2g(\nabla_vX,\nabla_vX) > 0 \).

  Let \( m \) be the minimum of the second derivative
  \( \ddot f(0;v,p') \) over the compact set \( S \) consisting of all
  \( (v,p') \) with \( p' \in M^X \cap \overline V \) and
  \( v \in (T_{p'}M^X)^\bot \).
  Now there is an \( \varepsilon > 0  \) such that
  \( \ddot f(t;v,p') > m/2  \) for all \( t < \varepsilon \)
  and all \( (v,p') \in S \).
  We then have \( \dot f(t) > tm/2 \) on
  \( (0,\varepsilon] \), so \( f \) is increasing.
  We also get \( f(t) > t^2m/4 \) on \( (0,\varepsilon] \).
  It follows that \( t < (2/\sqrt m)\norm {X_{\gamma(t)}} \),
  and \( \sup_{s \in [0,t]}
  \norm{d(\mu\circ \gamma)_{\gamma(s)}}
  \leqslant \sqrt3 \norm{X_{\gamma(t)}}
  \),
  whenever \( t \in (0,\varepsilon] \).

  Let \( U \) be an open neighbourhood of~\( p \) contained in the set
  of points \( \Set{f(t;v,p')}{(v,p') \in S, \abs t < \varepsilon} \).
  Then for each \( q \in U \), we now get
  from~\eqref{eq:rough-estimate}
  that
  \begin{equation*}
    \norm{\mu(q) - \mu(0)} \leqslant c\norm{X_q}^2
  \end{equation*}
  for \( c = \sqrt{12/m} \), as claimed.
\end{proof}

Suppose \( (M,X,\mu) \) is a hyperKähler manifold with tri-Hamiltonian
circle action generated by~\( X \) and with hyperKähler moment
map~\( \mu \).

We say a subset \( E \subset M \) is \emph{moment-polar},
if for each \( p \in M \) there is are open neighbourhoods \( U \)~of
\( p \in M \) and \( V \)~of \( p_0 = \mu(p) \in \bR^3 \) such that
\( U \subset \mu^{-1}(V) \) and
\( E \cap U \subset \mu^{-1}(D) \cap U \) for some polar set
\( D \subset V \). 
A function~\( h \) on \( M\setminus E \) will be of \emph{pull-back
type} if each \( p\in M \) has a neighbourhood~\( U \) such that the
restriction of \( h \) to \( U\setminus E \) is the pull-back under
\( \mu \) of a function in~\( \bR^3 \).
A necessary condition for \( h \) to be of pull-back type is that
\( dh \in \Span{\alpha_I,\alpha_J,\alpha_K} \).
This condition is sufficient on most points of~\( M \).
Indeed if \( p \) has neighbourhood~\( U \) such that the generic
fibres of~\( \mu \) in~\( U \) are connected, then such an \( h \) is
of pull-back type in~\( U \).
Any regular point of~\( \mu \) has such a neighbourhood, as does any
point of a codimension~\( 4 \) component of~\( M^X \), by the proof
of~\cref{prop:open}.
In the flat case, points of higher-dimensional fixed-point sets also
have this property, so \( h \) is of pull-back type on all of~\( M \).

Suppose \( (M,X,\mu) \) and \( (\check M,\check X,\check \mu) \) are
two such hyperKähler manifolds with tri-Hamiltonian circle actions.
We will say that \( M \) and \( \check M \) are
\emph{moment-harmonically equivalent} if there are open subsets
\( M_0 \subset M \), \( \check M_0 \subset \check M \) such that
\begin{enumerate}
\item the hyperKähler geometries on \( M_0 \) and \( \check M_0 \) are
  related by a combination of an elementary deformation and a twist as
  in \cref{thm:hK-twist} with \( h \) and \( \check h \) of pull-back
  type, and that 
\item \( E = M\setminus M_0 \) and
  \( \check E = \check M\setminus \check M_0 \) are moment-polar
  subsets.
\end{enumerate}

\begin{theorem}
  \label{thm:local-twist}
  Let  \( (M,X,\mu) \) and \( (\check M,\check X,\check \mu) \) be two
  complete moment-harmonically equivalent tri-Hamiltonian hyperKähler
  manifolds. 
  Let \( h\colon M_0 \to \bR \) be the deforming function, so that
  \( \check g \) is \Hrelated to \( \tilde g = g + h\gHX \).
  Then for each \( p \in M \), there is a
  neighbourhood~\( U \) on which \( h = \mu^*\psi \) where
  \begin{equation}
    \label{eq:local-sigma}
    \psi(q) = \frac\sigma{2\norm{q-p_0}} + \phi(q)
  \end{equation}
  with \( p_0 = \mu(p) \), \( \sigma \in \{-1,0,1\} \)
  and \( \phi \) harmonic on~\( \mu(U) \subset \bR^3 \).
  If \( \sigma = -1 \), then \( p \) is a fixed point of~\( X \),
  the component~\( C \) of the fixed point set through~\( p \) is of
  codimension~\( 4 \) and \( C \)~is also a component of a level set
  of~\( \mu \).  
\end{theorem}

In other words,
locally the singular behaviour of such moment-harmonic equivalence is
given by the hyperKähler modification.

\begin{proof}
  For \( p \in M_0 \), 
  the function~\( h \) is the pull-back of a harmonic function
  under~\( \mu \) in a neighbourhood of~\( p \) and we have the
  required relation with \( \sigma = 0 \).

  The geometries on \( M_0 \) and \( \check M_0 \) are related by
  \cref{prop:dual}, where we may take~\( \lambda = -1 \).
  The proof of that \lcnamecref{prop:dual}, shows that
  \( d\check \mu_I = \check \alpha_I \Hrel \alpha_I = d\mu_I \), etc.
  By adding a constant to~\( \check \mu \), we may thus assume that
  \( \check\mu \Hrel \mu \).
  The deforming function \( \check h \) from \( \check M_0 \) to
  \( M_0 \) is then \Hrelated to~\( -h \).

  As \( \check g \Hrel \tilde g = g + h\gHX \) is positive
  definite, we have \( h > -1/\norm X^2 \) on~\( M_0 \).
  Conversely for
  \( g \Hrel \tilde{\check g} = \check g + \check h\gHXc \)
  to be positive definite,
  we need \( h \Hrel -\check h < 1/\norm{\check X}^2 \) on
  \( \check M_0 \).

  Let \( p \in E = M \setminus M_0 \).  There is a neighbourhood
  \( U \) of \( p \) and a polar set \( D \) in \( \mu(U) \) such that
  \( E \cap U \subset \mu^{-1}(D) \cap U \).
  Note that \( \mu(E\cap U) \subset D \),
  so we can replace \( D \) by \( \mu(E\cap U) \).
  Also, for each compact set \( C \subset U \), we have \( E \cap C \)
  is compact, so \( \mu(E\cap C) = D \cap \mu(C) \) is compact, and in
  particular closed.
  Shrinking \( U \) we may thus take \( D \) to be closed in~\( \mu(U) \).

  Let us say that a point \( p \in E = M \setminus M_0 \) is
  \emph{strongly accessible} if there is a
  \( q \in \mu(U) \setminus D \) such that \( p_0 = \mu(p) \) is a
  closest point in~\( D \) to~\( q \) and the line segment \( [q,p_0] \)
  lies in \( \mu(U) \).
  Note that \( E \) always has strongly accessible points.
  Indeed for any \( p \in E \), each compact neighbourhood of~\( p \)
  contains a strongly accessible point of~\( E \).

  Let \( p \in E \) be strongly accessible.
  Because \( \mu \) is analytic, it follows that there is a smooth curve
  \( \gamma\colon [0,1] \to U \subset M \) with \( \gamma(1) = p \) and
  \( \mu(\gamma[0,1]) = [q,p_0] \subset \mu(U) \).
  We then have \( \gamma[0,1) \subset M_0 \).

  We may lift \( \gamma|_{[0,1)} \) horizontally to the twist
  bundle~\( P \) and project to~\( \check M_0 \), to get a smooth
  curve \( \check \gamma\colon [0,1) \to \check M_0 \).
  Say that \( \check \gamma \) is \Hrelated to~\( \gamma \).
  The length of \( \check\gamma \) is given by
  \begin{equation*}
    \ell(\check\gamma)
    = \int_0^1 \tilde g(\dot\gamma,\dot\gamma)^{1/2}\, dt
    \leqslant \int_0^1 \norm{\dot\gamma}(1+h\norm X^2)^{1/2}\, dt
  \end{equation*}

  Shrinking \( U \) if necessary, we may assume that \( h \) is a
  pull-back on \( U_0 = U\setminus E \) and that \( M^X \cap U \)
  is connected.
  If \( p \notin M^X \), we may take \( M^X \cap U = \varnothing \),
  then \( h \) is bounded below on \( U_0 \).
  If \( p \in M^X \), then \cref{lem:estimate} implies
  \( h(p') > -1/\norm X^2 \geqslant - 1/(c\norm{\mu(p')-p_0}) \) on
  \( U_0 \setminus M^X \), so
  \( h_+ = h + 1/(c\norm{\mu-p_0}) \) is bounded below on
  \( U_0' = U_0 \setminus M^X \).
  In both cases there are then positive constants \( c_3 \), \( c_4 \)
  and \( c_5 \) such that \( h_+ = c_3 + c_4h +
  c_5/\norm{\mu-p_0} > 0 \)  on \( U_0 \setminus M^X \) and 
  \begin{equation*}
    \ell(\check\gamma) \leqslant \int_0^1 h_+(\gamma(t))^{1/2}\,dt.
  \end{equation*}

  Now \( h_+ \) is the pull-back of a positive
  harmonic function on \( \mu(U_0) = \mu(U) \setminus D \),
  where \( D \) is polar and closed in~\( \mu(U) \), and so there is
  an estimate 
  \begin{equation*}
    h_+(\gamma(t))
    \leqslant c_1 + \frac{c_2}{\dist(\mu(\gamma(t)),D)}
    = c_1 + \frac{c_2}{\norm{\mu(\gamma(t))-p_0}}
  \end{equation*}
  similar to \eqref{eq:growth}.
  It follows that \( \check\gamma \) has finite length.
  By completeness of \( (\check M,\check g) \),
  there is a limit point
  \( \check p = \lim_{t\nearrow 1}\check\gamma(t) \in \check M \).

  Choose a similar type of neighbourhood \( \check U \)
  to~\( \check p \),
  so that the zero set of \( \check X \) is either disjoint
  from \( \check U \) or meets it in a single component passing
  through~\( \check p \).
  We may shrink \( U \) and \( \check U \) if necessary so that the
  pull-backs of \( U_0 \) and \( \check U_0 \) to the twist bundle
  \( P \) agree. 
  On \( U'_0 \) we then have
  \begin{equation}
    \label{eq:h-est}
    - \frac1{c\norm{\mu - p_0}} < - \frac1{\norm X^2} < h \Hrel -
    \check h < \frac1{\norm{\check X}^2} < \frac1{\check c\norm{\mu -
    p_0}}
  \end{equation}
  for some positive numbers \( c \) and~\( \check c \).

  If \( h \) is bounded on \( U'_0 \) then it has a continuous
  moment-harmonic extension to~\( U \) of the desired form with
  \( \sigma = 0 \) in~\eqref{eq:local-sigma}. 
  If \( h \) is not bounded, then
  \( h_+ = h + \frac1{c\norm{\mu-p_0}} \) is positive moment-harmonic
  on~\( U'_0 \).
  Furthermore, away from \( \mu^{-1}(p_0) \), it is bounded above.
  As \( E \cap U \)~is moment-polar, the function \( h_+ \) extends to
  a positive moment-harmonic function on
  \( U\setminus \mu^{-1}(p_0) \).
  Then by Bôcher's Theorem the extension has singular part a constant
  times \( \norm{\mu-p_0}^{-1} \).
  In other words, \( h \)~has the desired form with \( \sigma \ne 0 \)
  in \eqref{eq:local-sigma}; however we still need to constrain the
  possible non-zero values of~\( \sigma \).

  Let us consider the case when \( \sigma < 0 \), the other case
  \( \sigma > 0 \) will follow similarly by considering \( \check h \)
  instead of~\( h \).
  As \( h \) is not bounded below, we have from~\eqref{eq:h-est} that
  \( \norm X = 0 \) on the component of the fibre of \( \mu \) through~\( p \).
  Since fixed-point sets of \( X \) are at least of
  codimension~\( 4 \), \cref{cor:level-dim} implies that the fibre
  of~\( \mu \) through~\( p \) coincides with the fixed point set and
  is of codimension~\( 4 \).

  In particular, the normal bundle of the fibre through~\( p \) is of
  dimension~\( 4 \).
  Using the exponential map as in the proof of \cref{prop:open},
  we get a four-dimensional slice~\( M_1 \) through~\( p \), which at
  the origin is equivariantly identified with \( \bH \)
  and~\( k \XH \).
  Effectivity of the circle action, implies that \( k = \pm 1 \).
  Replacing \( X \) by \( -X \) if necessary, we may assume
  \( k = +1 \). 
  For a small distance sphere in \( M_1 \) centred at~\( p \),
  a connection one-form for the circle fibration via~\( \mu \) to
  \( S \subset \bR^3\setminus\{0\} \) is given by
  \( \beta_0 = \alpha/\norm X^2 \).
  This has curvature \( d\beta_0 = \norm X^{-2}d\alpha_0 + \norm
  X^{-4}(X\hook d\alpha_0)\wedge\alpha_0 \).

  On the other side of the twist, \( \check \alpha_0 \Hrel \alpha_0 \)
  and \( \norm{\check X}^2 \Hrel \norm X^2/(1+h\norm X^2) \) imply
  \( \check \beta_0 \Hrel \alpha_0(h + \norm X^{-2}) = \beta_0 +
  h\alpha_0 \).
  The curvature form is then
  \begin{equation*}
    \begin{split}
      d\check\beta_0
      &\Hrel (d-a^{-1}F\wedge X\hook) (\beta_0 + h\alpha_0)\\
      &= d\beta_0 + d(h\alpha_0)
      - \frac1{1+h\norm X^2}(d(h\alpha_0)+\Hodge_3dh)(1+h\norm X^2)\\
      &= d\beta_0 - \Hodge_3dh.
    \end{split}
  \end{equation*}
  Thus the change in the Chern class is given by
  \( - \frac1{2\pi}\int_S \Hodge_3 dh = \frac{\sigma}{4\pi} \int_S
  \Hodge d(\norm{p_0-q}^{-1}) \).
  Initially the Chern class is~\( +1 \); the only possible values
  after the twist are \( +1 \) or~\( 0 \).
  We have \( \sigma < 0 \), so the change is non-trivial.
  It follows that \( h \) changes the Chern class by~\( -1 \), and
  thus \( \sigma = -1 \), as claimed.

  The above describes~\( h \) in a neighbourhood~\( U_p \) of a
  strongly accessible point~\( p \) of~\( E \).
  We will say that \( p \in M \) is \emph{good} if \( h \) has
  the desired form~\eqref{eq:local-sigma} in a neighbourhood~\( U \)
  of~\( p \) and for each regular curve~\( \gamma \) in~\( U \)
  through~\( p \) there is a regular curve~\( \check\gamma \)
  on~\( \check M \), such that each segment of \( \check\gamma \) in
  \(\check M_0 \) is \Hrelated to the corresponding segment
  of~\( \gamma \).
  Let \( A \) be the set of good points of~\( M \).
  Then \( A \) is open, contains \( M_0 \) and all 
  strongly accessible points of~\( E \).
  
  Suppose \( A \) is not all of~\( M \).
  Fix \( x \in M_0 \) and let \( p \) be a boundary point of~\( A \)
  closest to~\( x \).
  When then have a minimal geodesic \( \gamma\colon [0,1] \to M \) with
  \( \gamma(0) = x \) and \( \gamma(1) = p \).
  This has \( \gamma[0,1) \subset A \).

  Choose a neighbourhood \( U \) of~\( p \) such that \( h \) is a
  pull-back on \( U\setminus E \) and which meets at most one
  connected component of~\( M^X \).
  If \( U\cap M^X \) is non-empty we may assume \( p \in M^X \).
  Then there are positive constants
  \( c_0 \) and \( c_1 \) such that
  \( h_+ = h + c_0 + c_1/\norm{\mu - p_0} \), for \( p_0 = \mu(p) \),
  is a positive moment-harmonic function on~\( U \setminus E \), cf.\
  derivation of the left-hand side of~\eqref{eq:h-est}.
  Now since \( h \) is given via~\eqref{eq:local-sigma} on~\( A \cap U \),
  the only singularities of \( h_+ \) on \( U \cap A \) are of
  the form \( \mu^*(+1/(2\norm{q-p_i})) \).
  Because \( h_+ \) take a finite value at some point of~\( U \),
  it follows that the \( p_i \) have no accumulation point.
  In particular, we may shrink~\( U \),
  so that \( h_+ \) is bounded on~\( U \cap A \).
  Replacing \( \gamma \) by a final segment lying in this new~\( U \),
  we get a curve in~\( A \)
  of finite length in the metric~\( \tilde g \) and ending at~\( p \).
  We may assume the initial point~\( y \) of this final segment lies
  in \( M_0 \), and choose a point \( z \) in~\( P \) lying
  over~\( y \).
  The curves \Hrelated to segments \( \gamma[0,1-\varepsilon] \) of
  \( \gamma \) may be uniquely specified by requiring them to start at
  the projection of~\( z \) to~\( \check M \).
  We thus get a curve
  \( \check \gamma \colon [0,1) \to \check M \) of finite length,
  and hence a limit point~\( \check p \in \check M \).
  Analysing this in the same way as for strongly accessible
  points, using the smoothness of \( \check M \) at~\( \check p \),
  we get the desired form for~\( h \) around~\( p \).
  We conclude that \( A \) is all of~\( M \),
  and thus that \( h \) has the claimed form everywhere. 
\end{proof}

Similar considerations show that locally any fixed-point set of
codimension~\( 4 \) may be twisted away after applying a Taub-NUT
deformation, so that \eqref{eq:unmod} is satisfied near the
fixed-point set.  
Further knowledge of the structure of the moment map~\( \mu \) is
required to prove global versions of \cref{thm:local-twist}.
For example, if it is known that the fibres
of~\( \mu \) are connected, then it follows that \( h \) is globally a
function of~\( \mu \) and \( h = \mu^*\psi \) with
\begin{equation*}
  \psi(q) = \phi(q) + \sum_{i \in I} \frac{\sigma_i}{2\norm{q-p_i}}
\end{equation*}
for some distinct \( p_i \in \mu(M) \), \( \sigma_i \in \{\pm 1\} \),
and \( \phi \) harmonic function on all of~\( \mu(M) \subset \bR^3 \).

\section{Strong HKT metrics}
\label{sec:strong-hkt-metrics}

A \emph{strong HKT} structure consists of an almost hyper-Hermitian
\( (M,g,I,J,K) \) such that the following two conditions hold:
\begin{gather}
  \label{eq:HKT}
  Id\omega_I = Jd\omega_J = Kd\omega_K,\\
  \label{eq:strong}
  dId\omega_I = 0.
\end{gather}
Condition~\eqref{eq:HKT} ensures that the manifold is hyperKähler with
torsion, HKT, the torsion three-form of the common Bismut connection
is \( c = -Id\omega_I \).  By \cite{Cabrera-S:aqH-torsion}, \( I \),
\( J \) and \( K \) are then necessarily integrable.
Condition~\eqref{eq:strong} is equivalent to the strong condition \(
dc=0 \).  

Starting with a hyperKähler manifold with a symmetry it is natural to
ask which elementary quaternionic deformations~\eqref{eq:elem-def}
twist to strong HKT structures.  

The following lemma used in the proof of the subsequent
\namecref{thm:SHKT-twist}, will establish the notion of rank of a
two-form.

\begin{lemma}
  \label{lem:wedge}
  Let \( \zeta \) and \( \eta \) be two forms an a vector space \( V
  \).  Write \( \ker\zeta=\Set{v\in V}{v\hook \zeta =0} \) and \(
  \rank \zeta = \dim V - \dim\ker \zeta \).  Suppose
  \begin{equation}
    \label{eq:zeta-eta}
    \zeta \wedge \eta = 0
  \end{equation}
  and that \( \zeta \ne 0 \).  If \( \rank \zeta > 4 \) then \( \eta =
  0 \).  If \( \rank \zeta = 4 \), then \( \rank \eta \leqslant 4 \) and \(
  \ker \eta \geqslant \ker \zeta \).
\end{lemma}

\begin{proof}
  Let \( U = (\ker\zeta)^\circ \) be the annihilator of \( \ker\zeta
  \) and choose a vector space complement~\( W \) to \( U \), so \(
  V^* = U\oplus W \).  Then \( \Lambda^2V^* = \Lambda^2 U \oplus
  U\wedge W \oplus \Lambda^2 W \) and we may correspondingly decompose
  \( \eta = \eta_U + \eta_m + \eta_W \).  Now \eqref{eq:zeta-eta},
  implies \( 0 = \zeta\wedge\eta_W \in \Lambda^2U\otimes\Lambda^2W \),
  so \( \eta_W = 0 \).  As \( \zeta \) is non-degenerate on \( \ker W \),
  the dimension of \( U \) is even and the map \(
  \zeta\wedge\cdot\colon \Lambda^kU \to \Lambda^{k+2}U \) is injective
  for \( k=1 \) if \( \dim U\geqslant 4 \) and \( k=2 \) if \( \dim U
  \geqslant 6 \).  But \eqref{eq:zeta-eta} implies \( 0 =
  \zeta\wedge\eta_m \in \Lambda^3U \otimes W \) and \( 0 =
  \zeta\wedge\eta_U \in \Lambda^4U \), from which the result follows.
\end{proof}

\begin{theorem}
  \label{thm:SHKT-twist}
  Let \( (M,g,I,J,K) \) be a hyperKähler manifold with a
  tri-holomorphic symmetry~\( X \) and dual one-form \( \alpha_0 =
  X^\flat \).  
  If the rank of~\( d\alpha_0 \) is at least~\( 12 \) then the only
  twists of elementary quaternionic deformations~\( \tilde g \) that
  are strong HKT are given by
  \begin{equation*}
    \tilde g = g + h\,\gHX,\quad a = \lambda \norm X^2, \quad F =
    \lambda d\alpha_0 
  \end{equation*}
  for some non-zero constant \( \lambda \) and some function \( h \)
  with \( dh \in \Span{\alpha_I,\alpha_J,\alpha_K} \) and harmonic: \(
  d\Hodge_3dh = 0 \).
\end{theorem}

In all cases, regardless of the rank of \( d\alpha_0 \), the above
twists are strong HKT.
They are not hyperKähler unless
\( d\alpha_0 = \norm X^2(dh\wedge\alpha_0 + \Hodge_3dh) \),
so \( X \) has rank~\( 4 \),
and \( h = c - 1/\norm X^2 \) for some constant~\( c \).
If \( X \) has rank~\( 4 \), then the distribution \( \bH X \) is
parallel, and so \( (M,g) \)~is reducible.

\begin{proof}
  We use the notation of \cref{sec:gener-hyperk-twists}.
  Let us write equation~\eqref{eq:dW-tomI} as
  \begin{equation*}
    d_W\tilde\omega_I = df \wedge \omega_I + dh \wedge
    (\alpha_{0I}+\alpha_{JK}) + hd\alpha_0 \wedge \alpha_I - \tilde h
    F\wedge \alpha_I, 
  \end{equation*}
  with \( \tilde h = (f+h\norm X^2)/a \).  For the twist to be
  hypercomplex we need \( F \) to be of complex type \( (1,1) \) for
  \( I \), \( J \) and \( K \), that is \( F \in \real{S^2E} \cong
  \sP(n) \subset \Lambda^2T^*M \), see \cite{Swann:twist}.  As \( X \)
  is a tri-holomorphic isometry, we also have \( d\alpha_0 \in
  \real{S^2E} \), so the HKT condition~\eqref{eq:HKT} becomes
  \begin{equation}
    \label{eq:HKT2}
    Idf \wedge \omega_I + Idh \wedge \omega_I^\alpha
    = Jdf \wedge \omega_J + Jdh \wedge \omega_J^\alpha = Kdf \wedge
    \omega_K + Kdh \wedge \omega_K^\alpha, 
  \end{equation}
  where \( \omega_I^\alpha = \alpha_{0I} + \alpha_{JK} \), etc.
  Decomposing with respect to the splitting \( TM = \bH X + V \), as
  above, we see \( Idf^{0,1}\wedge\omega_I^{0,2} =
  Jdf^{0,1}\wedge\omega_J^{0,2} \).  As \( \dim V \geqslant 8 \), this
  implies \( df^{0,1} = 0 \).  Considering the \( (1,2) \)-component
  of~\eqref{eq:HKT2}, then gives \( df = 0 \).  Scaling by a homothety
  we may thus take \( f \equiv 1 \).  It follows that \( dh^{0,1} = 0
  \), and as \( Xh = 0 \), we have \( dh = h_I\alpha_I + h_J\alpha_J +
  h_K\alpha_K \), for some functions \( h_I \), \( h_J \) and \( h_K
  \).

  We now write
  \begin{equation*}
    d_W\tilde\omega_I = dh \wedge \omega_I^\alpha +
    hd\alpha_0 \wedge \alpha_I - \frac1a(1+h\norm X^2) F\wedge\alpha_I
  \end{equation*}
  giving
  \begin{equation*}
    c_W \Hrel c = -Id_W\tilde\omega_I = (\Hodge_3 dh + hd_W\alpha_0 - \frac1a F)
    \wedge \alpha_0 
  \end{equation*}
  The strong condition~\eqref{eq:strong} is
  \begin{equation*}
    0 = d_Wc = d\Hodge_3dh\wedge\alpha_0 + (d_W(h\alpha_0) - \frac1aF
    + \Hodge_3 dh)\wedge d_W\alpha_0.
  \end{equation*}
  We may rewrite this equation as
  \begin{equation}
    \label{eq:h-zeta-eta}
    0 = \Delta h \vol_\alpha + \zeta \wedge \eta,
  \end{equation}
  where \( \zeta = (dh\wedge\alpha_0 + \Hodge_3dh) + hd\alpha_0 -
  \frac1a(1+h\norm X^2)F \), \( \eta = d\alpha_0 - \frac1a\norm X^2 F
  \), \( \vol_\alpha = \alpha_{0IJK} \) is \( \Delta h \) is the
  three-dimensional Laplacian.  Note that \( \zeta \) and \( \eta \)
  lie in \( [S^2E] \), and that, by the arguments after
  equation~\eqref{eq:aFhX}, \( \zeta = 0 \) if and only if the twist
  is hyperKähler.  We will therefore assume \( \zeta \ne 0 \).  Note
  also \( \zeta,\eta \in [S^2E] \) implies that their ranks are
  multiples of~\( 4 \).

  If \( h \) is harmonic, then we have \( \zeta \wedge \eta = 0 \).
  Applying \cref{lem:wedge}, we get that either \( \eta = 0 \) or \(
  \zeta \) and \( \eta \) both have rank~\( 4 \), with common kernel.
  In the latter case, we may write \( d\alpha_0 \) as a linear
  combination of \( dh\wedge\alpha_0+\Hodge_3dh \), \( \zeta \) and~\(
  \eta \).  As \( \zeta \) and \( \eta \) have the same kernel, this
  implies that \( d\alpha_0 \) has rank at most \( 8 \), contradicting
  our assumption.

  If \( \eta = 0 \), then we have \( d\alpha_0  = \frac 1a\norm X^2 F
  \).  Contracting with \( X \), gives \( d\log a = d\log\norm X^2 \),
  so \( a = \lambda\norm X^2 \) and \( F = \lambda d\alpha_0 \) for
  some constant~\( \lambda \).

  Suppose now that \( h \) is not harmonic.  Then \(
  \zeta^{2,0}\wedge\eta^{2,0}= - \Delta h\vol_\alpha \ne 0 \), implies
  that \( \zeta^{2,0} \) and \( \eta^{2,0} \) are both non-zero, and
  hence non-degenerate forms on~\( \bH X \).  Note that they have the
  opposite orientation to~\( \omega_I^\alpha \).  

  The \( (3,1) \)-equation \( \zeta^{1,1}\wedge\eta^{2,0} +
  \zeta^{2,0}\wedge\eta^{1,1} = 0 \) from~\eqref{eq:h-zeta-eta}
  implies that \( \eta^{1,1} \) is uniquely determined by \(
  \zeta^{1,1} \) and that the intersections of their kernels with~\( V
  \) are the same space~\( V' \).

  Suppose \( \zeta^{1,1} = 0 \), then \( \eta^{1,1} = 0 \) too and the
  \( (2,2) \)-component of~\eqref{eq:h-zeta-eta} is \(
  \zeta^{2,0}\wedge\eta^{0,2} + \zeta^{0,2}\wedge\eta^{2,0} = 0 \).
  This means that \( \zeta^{0,2} \) and \( \eta^{0,2} \) have the same
  rank and the same kernel.  As \( d\alpha_0 \) has rank at least four
  on~\( V \), one of these forms is non-zero, and so both are.  But
  the \( (0,4) \)-component says \( \zeta^{0,2} \wedge \eta^{0,2} = 0
  \).  By \cref{lem:wedge}, we have that both these forms have rank~\(
  4 \) and conclude that the rank of \( d\alpha_0 \) is at most~\( 8
  \), contradicting our assumption.

  We have thus shown that \( \zeta^{1,1} \) and \( \eta^{1,1} \) are
  non-zero.  Considering the \( (1,3) \)-component \(
  \zeta^{1,1}\wedge\eta^{0,2}+\zeta^{0,2}\wedge\eta^{1,1} = 0 \)
  of~\eqref{eq:h-zeta-eta}, we find that \( \zeta^{2,0} \) and \(
  \eta^{2,0} \) have the same kernel.  The \( (0,4) \)-component of
  the equation implies that \( \zeta^{2,0} \) and \( \eta^{2,0} \)
  have rank at most~\( 4 \) and then the \( (2,2) \)-component implies
  that their common kernel contains~\( V' \).  It follows that \(
  \zeta \) and \( \eta \) both have rank at most~\( 8 \) and that the
  same is true of \( d\alpha_0 \), which again is a contradiction.
\end{proof}

Note that as \( d\alpha_0 \in [S^2E] \), the kernel is not only
integrable but also quaternionic.  By \cite{Gray:Sp} this means that
the corresponding leaves are totally geodesic.

The above twists have \( F \) exact.  They exist whenever \( X \) has
no zeros.  Simple examples are therefore provided by taking \( M \)
homogeneous.  As hyperKähler manifolds are Ricci flat, homogeneous
examples are flat \cite{Alekseevskii-K:Ricci}.  However,
\cite{Barberis-DF:solvable} have classified the left-invariant
hyperKähler metrics on Lie groups~\( G \).  They find that \( G \) is
necessarily two-step solvable and give an explicit structural
description.  Taking \( X \) to be any vector field generated by the
left action and \( h \) to be constant, then provides strong HKT
structures on~\( G \).

Other examples are provided by considering for compact \( G \) the
hyperKähler metrics on \( TG^{\bC} \) constructed by
Kronheimer~\cite{Kronheimer:cotangent} (see also
\cite{Dancer-Swann:compact-Lie}).  These carry a tri-holomorphic
action of \( G\times G \), with each factor acting freely.

With this twist we have on the strong HKT side an isometry generated
by the vector field \Hrelated to \( -\frac1aX = - X/(\lambda\norm X^2) \).
The corresponding dual one-form is
\begin{equation*}
  \alpha_0^W \Hrel -\frac1{\lambda\norm X^2}X \hook \tilde g
  =  - \frac1{\lambda\norm X^2} \alpha_0
  - \frac h{\lambda\norm X^2}\norm X^2\alpha_0
  = - \frac{1+h\norm X^2}{\lambda\norm X^2}\alpha_0.
\end{equation*}
This has exterior derivative
\begin{equation*}
  \begin{split}
    d\alpha_0^W
    &\Hrel - (d- \frac1aF\wedge X\hook) \Bigl(\frac{1+h\norm X^2}{\lambda\norm X^2}\alpha_0\Bigr)\\
    &=- \frac{1+h\norm X^2}{\lambda\norm X^2} d\alpha_0
    - \frac1\lambda(dh - \norm X^{-4}d\norm X^2) \wedge \alpha_0
    + \frac1{\lambda\norm X^2}\lambda d\alpha_0 \, \norm X^2\\
    &=- \frac{1+(h-\lambda)\norm X^2}{\lambda\norm X^2} d\alpha_0
    - \frac1\lambda(dh + \norm X^{-4}X \hook d\alpha_0) \wedge \alpha_0,
  \end{split}
\end{equation*}
which is proportional to \( d\alpha_0 \) plus a decomposable term.  In
particular, \( d\alpha_0^W \) does not lie in \( [S^2E] \),
unless \( h + \norm X^{-2} \) is constant.


\begin{thebibliography}{10}

\bibitem{Alekseevskii-K:Ricci}
D.~V. Alekseevski\u\i{} and B.~N. Kimel'fel'd, \emph{Structure of homogeneous
  {R}iemannian spaces with zero {R}icci curvature}, Funktsional. Anal. i
  Prilozhen. \textbf{9} (1975), no.~2, 5--11.

\bibitem{AndersonMT-KL:infinite}
Michael~T. Anderson, Peter~B. Kronheimer, and Claude LeBrun, \emph{Complete
  {R}icci-flat {K\"a}hler manifolds of infinite topological type}, Comm. Math.
  Phys. \textbf{125} (1989), no.~4, 637--642.

\bibitem{Armitage-G:potential}
David~H. Armitage and Stephen~J. Gardiner, \emph{Classical potential theory},
  Springer Monographs in Mathematics, Springer-Verlag London, Ltd., London,
  2001.

\bibitem{Axler-BR:Bocher}
Sheldon Axler, Paul Bourdon, and Wade Ramey, \emph{B\^{o}cher's theorem}, Amer.
  Math. Monthly \textbf{99} (1992), no.~1, 51--55.

\bibitem{Barberis-DF:solvable}
M.~L. Barberis, I.~Dotti, and A.~Fino, \emph{Hyper-{K\"a}hler quotients of
  solvable {L}ie groups}, J. Geom. Phys. \textbf{56} (2006), no.~4, 691--711.

\bibitem{Barberis-F:strong}
M.~L. Barberis and A.~Fino, \emph{New {HKT} manifolds arising from quaternionic
  representations}, Math. Z. \textbf{267} (2011), no.~3-4, 717--735.

\bibitem{Besse:Einstein}
A.~L. Besse, \emph{Einstein manifolds}, Ergebnisse der Mathematik und ihrer
  Grenzgebiete, 3. Folge, vol.~10, Springer, 1987.

\bibitem{Bielawski:tri-Hamiltonian}
R.~Bielawski, \emph{Complete hyper-{K\"a}hler $4n$-manifolds with a local
  tri\hyphen hamiltonian {$\mathbb R\sp n$}-action}, Math. Ann. \textbf{314}
  (1999), no.~3, 505--528.

\bibitem{Bierstone-M:semianalytic}
Edward Bierstone and Pierre~D. Milman, \emph{Semianalytic and subanalytic
  sets}, Inst. Hautes \'{E}tudes Sci. Publ. Math. (1988), no.~67, 5--42.
  \MR{972342 (89k:32011)}

\bibitem{Dancer-Swann:compact-Lie}
A.~S. Dancer and A.~F. Swann, \emph{Hyper{K\"a}hler metrics associated to
  compact lie groups}, Math. Proc. Camb. Phil. Soc. \textbf{120} (1996),
  61--69.

\bibitem{Dancer-S:mod}
\bysame, \emph{Modifying hyperk{\"a}hler manifolds with circle symmetry}, Asian
  J. Math. \textbf{10} (2006), no.~4, 815--826.

\bibitem{Dancer-S:cuts}
\bysame, \emph{Non-{A}belian cut constructions and hyperk{\"a}hler
  modifications}, Rend. Sem. Mat. Univ. Pol. Torino \textbf{68} (2010), no.~2,
  157--170.

\bibitem{Gibbons-H:multi}
G.~W. Gibbons and S.~W. Hawking, \emph{Gravitational multi-instantons}, Phys.
  Lett. \textbf{B78} (1978), 430--432.

\bibitem{Gibbons-PS:hkt-okt}
G.~W. Gibbons, G.~Papadopoulos, and K.~S. Stelle, \emph{{HKT} and {OKT}
  geometries on soliton black hole moduli spaces}, Nuclear Phys. B \textbf{508}
  (1997), no.~3, 623--658.

\bibitem{Goto:A-infinity}
R.~Goto, \emph{On hyper-{K\"a}hler manifolds of type {$A_\infty$}}, Geometric
  and Funct. Anal. \textbf{4} (1994), 424--454.

\bibitem{Grantcharov-P:HKT}
G.~Grantcharov and Y.~S. Poon, \emph{Geometry of hyper-{K\"a}hler connections
  with torsion}, Comm. Math. Phys. \textbf{213} (2000), no.~1, 19--37.

\bibitem{Gray:Sp}
A.~Gray, \emph{A note on manifolds whose holonomy group is a subgroup of
  $sp(n)\cdot sp(1)$}, Michigan Math.~J. \textbf{16} (1969), 125--128, Errata
  \textbf{17} (1970), 409.

\bibitem{Hattori:Ainfty-volume}
Kota Hattori, \emph{The volume growth of hyper-{K\"a}hler manifolds of type
  {$A_\infty$}}, Journal of Geometric Analysis \textbf{21} (2011), no.~4,
  920--949.

\bibitem{Hattori:c-symplectic-Ainfty}
\bysame, \emph{The holomorphic symplectic structures on hyper-{K\"a}hler
  manifolds of type {$A_\infty$}}, Adv. Geom. \textbf{14} (2014), no.~4,
  613--630.

\bibitem{Hawking:gravitational}
S.~W. Hawking, \emph{Gravitational instantons}, Phys. Lett. A \textbf{60}
  (1977), no.~2, 81--83.

\bibitem{Hitchin:Montreal}
N.~J. Hitchin, \emph{Monopoles, minimal surfaces and algebraic curves}, Les
  presses de l'Universit\'{e} de Montr\'{e}al, 1987.

\bibitem{Hitchin-KLR:hK}
Nigel~J. Hitchin, A.~Karlhede, U.~Lindstr\"{o}m, and M.~Ro\v{c}ek,
  \emph{Hyper{K\"a}hler metrics and supersymmetry}, Comm. Math. Phys.
  \textbf{108} (1987), 535--589.

\bibitem{Joyce:hypercomplex}
D.~Joyce, \emph{Compact hypercomplex and quaternionic manifolds},
  J.~Differential Geom. \textbf{35} (1992), 743--761.

\bibitem{Kronheimer:cotangent}
P.~B. Kronheimer, \emph{A hyper{K\"a}hler structure on the cotangent bundle of
  a complex lie group}, 1986.

\bibitem{Macia-S:c-map}
\'{O}scar Maci\'{a} and Andrew~F. Swann, \emph{Twist geometry of the c-map},
  Comm. Math. Phys. \textbf{336} (2015), no.~3, 1329--1357.

\bibitem{Martin:minimal}
Robert~S. Martin, \emph{Minimal positive harmonic functions}, Trans. Amer.
  Math. Soc. \textbf{49} (1941), 137--172.

\bibitem{Cabrera-S:aqH-torsion}
F.~Mart\'\i{}n~Cabrera and A.~F. Swann, \emph{The intrinsic torsion of almost
  quaternion\hyphen {H}ermitian manifolds}, Ann. Inst. Fourier (Grenoble)
  \textbf{58} (2008), no.~5, 1455--1497.

\bibitem{Milnor:singular}
John Milnor, \emph{Singular points of complex hypersurfaces}, Annals of
  Mathematics Studies, No. 61, Princeton University Press, Princeton, N.J.;
  University of Tokyo Press, Tokyo, 1968. \MR{0239612 (39 \#969)}

\bibitem{Minerbe:Rigidity-Taub-NUT}
Vincent Minerbe, \emph{Rigidity for multi-{T}aub-{NUT} metrics}, J. Reine
  Angew. Math. \textbf{656} (2011), 47--58.

\bibitem{Naim:Martin-boundary}
Linda Na\"\i{}m, \emph{Sur le r\^{o}le de la fronti\`{e}re de {R}. {S}.
  {M}artin dans la th\'{e}orie du potentiel}, Ann. Inst. Fourier, Grenoble
  \textbf{7} (1957), 183--281.

\bibitem{Schoen-Y:lectures}
R.~Schoen and S.-T. Yau, \emph{Lectures on differential geometry}, Conference
  Proceedings and Lecture Notes in Geometry and Topology, I, International
  Press, Cambridge, MA, 1994, Lecture notes prepared by Wei Yue Ding, Kung
  Ching Chang [Gong Qing Zhang], Jia Qing Zhong and Yi Chao Xu, Translated from
  the Chinese by Ding and S. Y. Cheng, Preface translated from the Chinese by
  Kaising Tso.

\bibitem{Spindel-STvP:complex}
Ph. Spindel, A.~Sevrin, W.~Troost, and A.~Van~Proeyen, \emph{Extended
  supersymmetric $\sigma$-models on group manifolds. {I}. the complex
  structures}, Nuclear Phys. B \textbf{308} (1988), no.~2-3, 662--698.

\bibitem{Swann:T}
A.~F. Swann, \emph{T is for twist}, Proceedings of the XV International
  Workshop on Geometry and Physics, Puerto de la Cruz, September 11--16, 2006
  (D.~Iglesias~Ponte, J.~C. Marrero~Gonz\'{a}lez, F.~Mart\'\i{}n~Cabrera,
  E.~Padr\'{o}n~Fern\'{a}ndez, and Sosa Mart\'\i{}n, eds.), Publicaciones de la
  Real Sociedad Matem\'{a}tica Espa\~{n}ola, vol.~11, Spanish Royal
  Mathematical Society, 2007, pp.~83--94.

\bibitem{Swann:twist}
\bysame, \emph{Twisting {H}ermitian and hypercomplex geometries}, Duke Math. J.
  \textbf{155} (2010), no.~2, 403--431.

\bibitem{Tod-W:sd-Killing}
K.~P. Tod and R.~S. Ward, \emph{Self-dual metrics with self-dual {K}illing
  vectors}, Proc. Roy. Soc. London Ser. A \textbf{368} (1979), no.~1734,
  411--427.

\bibitem{Wall:regular}
C.~T.~C. Wall, \emph{Regular stratifications}, Dynamical systems---Warwick 1974
  (Proc. Sympos. Appl. Topology and Dynamical Systems, Univ. Warwick, Coventry,
  1973/1974; presented to E. C. Zeeman on his fiftieth birthday), Springer,
  Berlin, 1975, pp.~332--344. Lecture Notes in Math., Vol. 468. \MR{0649271 (58
  \#31184)}

\end{thebibliography}

\providecommand{\bysame}{\leavevmode\hbox to3em{\hrulefill}\thinspace}
\providecommand{\MR}{\relax\ifhmode\unskip\space\fi MR }
\providecommand{\MRhref}[2]{%
  \href{http://www.ams.org/mathscinet-getitem?mr=#1}{#2}
}
\providecommand{\href}[2]{#2}

{\small
  \parindent0pt\parskip\baselineskip

  A. F. Swann

  Department of Mathematics, Aarhus University, Ny Munkegade 118, Bldg
  1530, DK-8000 Aarhus C, Denmark.

  \textit{E-mail}: \url{swann@math.au.dk}
\par}

\end{document}